\pdfoutput=1
\RequirePackage{ifpdf}
\ifpdf % We~are running pdfTeX in pdf mode
\documentclass[pdftex]{sigma}
\else
\documentclass{sigma}
\fi

\numberwithin{equation}{section}

\newtheorem{Theorem}{Theorem}[section]
\newtheorem*{Theorem*}{Theorem}
\newtheorem{Corollary}[Theorem]{Corollary}
\newtheorem{Lemma}[Theorem]{Lemma}
\newtheorem{Proposition}[Theorem]{Proposition}
 { \theoremstyle{definition}

\newtheorem{Remark}[Theorem]{Remark} }

\newcommand{\field}[1]{\mathbb{#1}}

\newcommand{\Z}{\field{Z}}

\begin{document}
%\allowdisplaybreaks

\newcommand{\arXivNumber}{2307.02346}

\renewcommand{\thefootnote}{}

\renewcommand{\PaperNumber}{032}

\FirstPageHeading

\ShortArticleName{Bilateral Bailey Lattices and Andrews--Gordon Type Identities}

\ArticleName{Bilateral Bailey Lattices\\ and Andrews--Gordon Type Identities\footnote{This paper is a~contribution to the Special Issue on Basic Hypergeometric Series Associated with Root Systems and Applications in honor of Stephen C.~Milne's 75th birthday. The~full collection is available at \href{https://www.emis.de/journals/SIGMA/Milne.html}{https://www.emis.de/journals/SIGMA/Milne.html}}}

\Author{Jehanne DOUSSE~$^{\rm a}$, Fr\'ed\'eric JOUHET~$^{\rm b}$ and Isaac KONAN~$^{\rm b}$}

\AuthorNameForHeading{J.~Dousse, F.~Jouhet and I.~Konan}

\Address{$^{\rm a)}$~Universit\'e de Gen\`eve, 7--9, rue Conseil G\'en\'eral, 1205 Gen\`eve, Switzerland}
\EmailD{\href{mailto:jehanne.dousse@unige.ch}{jehanne.dousse@unige.ch}}
\URLaddressD{\url{https://www.unige.ch/~doussej/}}

\Address{$^{\rm b)}$~Univ Lyon, Universit\'e Claude Bernard Lyon 1, UMR5208, Institut Camille Jordan,\\
\hphantom{$^{\rm b)}$}~69622 Villeurbanne, France}
\EmailD{\href{jouhet@math.univ-lyon1.fr}{jouhet@math.univ-lyon1.fr}, \href{konan@math.univ-lyon1.fr}{konan@math.univ-lyon1.fr}}
\URLaddressD{\url{https://math.univ-lyon1.fr/~jouhet/}}

\ArticleDates{Received October 15, 2024, in final form April 11, 2025; Published online April 29, 2025}

\Abstract{We show that the Bailey lattice can be extended to a bilateral version in just a~few lines from the bilateral Bailey lemma, using a very simple lemma transforming bilateral Bailey pairs relative to $a$ into bilateral Bailey pairs relative to $a/q$. Using this and similar lemmas, we give bilateral versions and simple proofs of other (new and known) Bailey lattices, including a Bailey lattice of Warnaar and the inverses of Bailey lattices of Lovejoy. As consequences of our bilateral point of view, we derive new $m$-versions of the Andrews--Gordon identities, Bressoud's identities, a new companion to Bressoud's identities, and the Bressoud--G\"ollnitz--Gordon identities. Finally, we give a new elementary proof of another very general identity of Bressoud using one of our Bailey lattices.}


\Classification{11P84; 05A30; 33D15; 33D90}

\renewcommand{\thefootnote}{\arabic{footnote}}
\setcounter{footnote}{0}

\section{Introduction and statement of results}\label{sec:intro}

A classical approach to obtain and prove $q$-series identities is the Bailey lemma, originally found by Bailey~\cite{Ba}, and whose iterative strength was later highlighted by Andrews~\cite{A1, A2, AAR} through the so-called Bailey chain.
Fix complex numbers $a$ and $q$. Recall~\cite{Ba} that a Bailey pair $((\alpha_n)_{n\geq0}, (\beta_n)_{n\geq0})$ ($(\alpha_n, \beta_n)$ for short) relative to $a$ is a pair of sequences satisfying
\begin{equation}\label{bp}
\beta_n=\sum_{j=0}^n\frac{\alpha_j}{(q)_{n-j}(aq)_{n+j}}\qquad\forall n\in\mathbb{N}.
\end{equation}
Here and throughout the paper, we use standard $q$-series notations which can be found in~\cite{GR}
\begin{equation*}
(a)_\infty = (a;q)_\infty:=\prod_{j\geq 0}\bigl(1-aq^j\bigr)\qquad\mbox{and}\qquad(a)_k = (a;q)_k:=\frac{(a;q)_\infty}{\bigl(aq^k;q\bigr)_\infty},
\end{equation*}
where $k \in \Z$, and
$
(a_1,\ldots,a_m)_k:=(a_1)_k\cdots(a_m)_k$,
where $k$ is an integer or infinity, and as usual~${|q|<1}$ to ensure convergence of infinite products.

The Bailey lemma describes how, given a Bailey pair, one can produce infinitely many of them. Originally, Bailey~\cite{Ba} stated the Bailey transform and applied
it in a number of cases without iterating it, to obtain (what
is now called) the weak Bailey lemma. Andrews~\cite{A1} reformulated and generalised Bailey's result to what is called the strong Bailey lemma or Bailey lemma, exhibiting its iterative nature, therefore giving rise to the concept of Bailey chain. Also Paule~\cite{P} independently noticed that Bailey's result could be iterated (using one extra parameter at each step instead of the two parameters used by Andrews, see below).
\begin{Theorem}[Bailey lemma]\label{thm:baileylemma}
If $(\alpha_n, \beta_n)$ is a Bailey pair relative to $a$, then so is $(\alpha'_n, \beta'_n)$, where
\begin{equation*}\alpha'_n=\frac{(\rho,\sigma)_n(aq/\rho\sigma)^n}{(aq/\rho,aq/\sigma)_n} \alpha_n \qquad\mbox{and}\qquad\beta'_n=\sum_{j=0}^n\frac{(\rho,\sigma)_j(aq/\rho\sigma)_{n-j}(aq/\rho\sigma)^j}{(q)_{n-j}(aq/\rho,aq/\sigma)_n} \beta_j.
\end{equation*}
\end{Theorem}
Despite its quite elementary proof, as it only requires the $q$-analogue of the Pfaff--Saalsch\"utz formula (see~\cite[formula~(II.12)]{GR}), which is itself consequence of the $q$-binomial theorem (or can alternatively be proved elementarily by induction), it yields many formulas in $q$-series, some of which are highly non trivial. For instance, in~\cite[equations (2.12) and (2.13)]{A1}, the following unit Bailey pair (relative to $a$) is considered (proving that it is indeed a Bailey pair is elementary, it can be done either directly or by inverting the relation~\eqref{bp})
\begin{equation}\label{ubp}
\alpha_n=(-1)^nq^{\binom{n}{2}}\frac{1-aq^{2n}}{1-a}\frac{(a)_n}{(q)_n},\qquad \beta_n=\delta_{n,0}.
\end{equation}
Applying Theorem~\ref{thm:baileylemma} twice to the unit Bailey pair~\eqref{ubp} yields a simple proof of the famous Rogers--Ramanujan identities \cite{RR19}.

\begin{Theorem}[Rogers--Ramanujan identities]\label{th:RR}
Let $i=0$ or $1$. Then
\begin{equation*}
 \sum_{n \geq 0} \frac{q^{n^2+ (1-i)n}}{(q)_n} = \frac{1}{\bigl(q^{2-i},q^{3+i};q^5\bigr)_{\infty}}.
\end{equation*}
\end{Theorem}

Iterating $r\geq 2$ times, this process yields the $i=1$ and $i=r$ special instances of the Andrews--Gordon identities \cite{A74}.

\begin{Theorem}[Andrews--Gordon identities]\label{th:AGseries}
Let $r \geq 2$ and $1 \leq i \leq r$ be two integers. We have
\begin{gather}\label{AG}
\sum_{s_1\geq\dots\geq s_{r-1}\geq0}\frac{q^{s_1^2+\dots+s_{r-1}^2+s_{i}+\dots+s_{r-1}}}{(q)_{s_1-s_2}\cdots(q)_{s_{r-2}-s_{r-1}}(q)_{s_{r-1}}}=\frac{\bigl(q^{2r+1},q^{i},q^{2r-i+1};q^{2r+1}\bigr)_\infty}{(q)_\infty}.
\end{gather}
\end{Theorem}
These identities are the analytic analogue of Gordon's partition theorem~\cite{Go}.

However it is not possible to prove the cases $1 < i < r$ of the Andrews--Gordon identities with only the Bailey chain. Thus the Bailey lattice was developed in~\cite{AAB} as a more general tool which enabled the authors to give a proof of the full Andrews--Gordon identities. The key point is to change the parameter $a$ to $a/q$ at some point before iterating the Bailey lemma, therefore providing a concept of Bailey lattice instead of the classical Bailey chain described above.

Here is the classical Bailey lattice proved in~\cite{AAB}. Its proof is only a little bit more involved than for Theorem~\ref{thm:baileylemma}, as it relies again on the above-mentioned $q$-Pfaff--Saalsch\"utz formula, together with the $q$-Chu--Vandermonde terminating $_2\phi_1$ summation~\cite[formula~(II.6)]{GR} and a~final mild division into two cases.

\begin{Theorem}[Bailey lattice]\label{thm:baileylattice}
If $(\alpha_n, \beta_n)$ is a Bailey pair relative to $a$, then $(\alpha'_n, \beta'_n)$ is a~Bailey pair relative to $a/q$, where
\begin{equation*}
\alpha'_0=\alpha_0, \qquad \alpha'_n=\frac{(\rho,\sigma)_n(a/\rho\sigma)^n}{(a/\rho,a/\sigma)_n}(1-a)\left(\frac{\alpha_n}{1-aq^{2n}}-\frac{aq^{2n-2}\alpha_{n-1}}{1-aq^{2n-2}}\right),
\end{equation*}
and
\begin{equation*}\beta'_n=\sum_{j=0}^n{(\rho,\sigma)_j(a/\rho\sigma)_{n-j}(a/\rho\sigma)^j\over (q)_{n-j}(a/\rho,a/\sigma)_n} \beta_j.
\end{equation*}
\end{Theorem}

Alternatively, Andrews, Schilling and Warnaar showed in~\cite[Section~3]{ASW} that it is possible to prove~\eqref{AG} using the Bailey lemma and bypassing the Bailey lattice: actually, their method is related to what we present below. Indeed at some point they use computations which are equivalent to a special case of the inverse of Lemma~\ref{lem:Mac}, which is the $b=0$ case of Lovejoy's Lemma~\ref{lem:love} (these lemmas did not exist at the time).

Note that Paule~\cite{P} derives~\eqref{AG} from his weaker version of the Bailey lemma. In~\cite{BIS}, it is also explained how a change of base allows one to avoid using the Bailey lattice. Recently, McLaughlin~\cite{M} showed that~\eqref{AG} can be proved much more easily by combining the classical Bailey Lemma with a simple lemma (see also the result of Lovejoy~\cite[Lemma 2.2]{Lo22} which corresponds to the case $a=q$ of McLaughlin's result).
\begin{Lemma}[McLaughlin]
\label{lem:Mac}
If $(\alpha_n, \beta_n)$ is a Bailey pair relative to $a$, then $(\alpha'_n, \beta'_n)$ is a Bailey pair relative to $a/q$, where
\begin{equation*}
\alpha'_0=\alpha_0, \qquad \alpha'_n=(1-a)\left(\frac{\alpha_n}{1-aq^{2n}}-\frac{aq^{2n-2}\alpha_{n-1}}{1-aq^{2n-2}}\right), \qquad \beta'_n=\beta_n.
\end{equation*}
\end{Lemma}

In this paper, we will show, among other things, that this lemma and the Bailey lattice can be extended to bilateral versions.

As noted in \cite{BMS} and~\cite{J}, it is possible to define for all $n\in\mathbb{Z}$ a \emph{bilateral Bailey pair} $(\alpha_n, \beta_n)$ relative to $a$ by the relation
\begin{equation}\label{bbp}
\beta_n=\sum_{j\leq n}\frac{\alpha_j}{(q)_{n-j}(aq)_{n+j}}\qquad\forall n\in\mathbb{Z}.
\end{equation}

\begin{Remark}\label{rk:bilatinversion}
The relation~\eqref{bp} defining classical (unilateral) Bailey pairs is a special instance of the above relation defining bilateral ones, as choosing $\alpha_n=0$ for negative integers $n$ in~\eqref{bbp} implies $\beta_n=0$ for $n$ negative. Actually, the converse is also true, as the classical Bailey inversion holds for bilateral Bailey pairs: $(\alpha_n, \beta_n)$ is a bilateral Bailey pair relative to $a$ if and only if
\begin{equation}\label{bilatinversion}
\alpha_n=\frac{1-aq^{2n}}{1-a}\sum_{j\leq n}\frac{(a)_{n+j}}{(q)_{n-j}}(-1)^{n-j}q^{\binom{n-j}{2}}\beta_j\qquad\forall n\in\mathbb{Z}.
\end{equation}
\textit{Thus, from all our results in this paper, one can deduce the corresponding unilateral results by setting $\alpha_n=0$ $($or equivalently $\beta_n=0)$ for all $n<0$.}
\end{Remark}

In \cite{BMS}, the Bailey lemma is extended in the following way.
\begin{Theorem}[bilateral Bailey lemma]\label{thm:bilatbaileylemma}
If $(\alpha_n, \beta_n)$ is a bilateral Bailey pair relative to $a$, then
so is $(\alpha'_n, \beta'_n)$, where
\begin{equation*}\alpha'_n=\frac{(\rho,\sigma)_n(aq/\rho\sigma)^n}{
(aq/\rho,aq/\sigma)_n}\alpha_n\qquad\mbox{and}\qquad\beta'_n=\sum_{j\leq n}\frac{(\rho,\sigma)_j(aq/\rho\sigma)_{n-j}(aq/\rho\sigma)^j}{(q)_{n-j}(aq/\rho,aq/\sigma)_n}\beta_j,
\end{equation*}
subject to convergence conditions on the sequences $\alpha_n$ and $\beta_n$, which make the relevant infinite series absolutely convergent.
\end{Theorem}

Our first result is an extension of the Bailey lattice to the bilateral case.

\begin{Theorem}[bilateral Bailey lattice]\label{thm:bilatbaileylattice}
If $(\alpha_n, \beta_n)$ is a bilateral Bailey pair relative to $a$, then~$(\alpha'_n, \beta'_n)$ is a bilateral Bailey pair relative to $a/q$, where
\begin{equation*}\alpha'_n=\frac{(\rho,\sigma)_n(a/\rho\sigma)^n}{(a/\rho,a/\sigma)_n}(1-a)\left(\frac{\alpha_n}{1-aq^{2n}}-\frac{aq^{2n-2}\alpha_{n-1}}{1-aq^{2n-2}}\right),
\end{equation*}
and
\begin{equation*}\beta'_n=\sum_{j\leq n}{(\rho,\sigma)_j(a/\rho\sigma)_{n-j}(a/\rho\sigma)^j\over (q)_{n-j}(a/\rho,a/\sigma)_n} \beta_j,
\end{equation*}
subject to convergence conditions on the sequences $\alpha_n$ and $\beta_n$, which make the relevant infinite series absolutely convergent.
\end{Theorem}

As mentioned above, several proofs have been given for the (unilateral) Bailey lattice. On the other hand, simpler proofs were given to prove the Andrews--Gordon identities without using the Bailey lattice. Here we give a very simple proof of our bilateral Bailey lattice, which when considering Bailey pairs such that $\alpha_n=0$ for $n<0$ reduces to the unilateral Bailey lattice. Hence we provide in particular a very simple proof of the classical Bailey lattice.

The key in our proof is the following simple lemma, which generalises McLaughlin's unilateral Lemma \ref{lem:Mac} and transforms bilateral Bailey pairs relative to $a$ into bilateral Bailey pairs relative to $a/q$.
\begin{Lemma}[key Lemma~1]\label{lem:key1}
If $(\alpha_n, \beta_n)$ is a bilateral Bailey pair relative to $a$, then $(\alpha'_n, \beta'_n)$ is a bilateral Bailey pair relative to $a/q$, where
\begin{equation}\label{eq:baileylattice1}
\alpha'_n=(1-a)\left(\frac{\alpha_n}{1-aq^{2n}}-\frac{aq^{2n-2}\alpha_{n-1}}{1-aq^{2n-2}}\right), \qquad \beta'_n=\beta_n,
\end{equation}
subject to convergence conditions on the sequences $\alpha_n$ and $\beta_n$, which make the relevant infinite series absolutely convergent.
\end{Lemma}
While it is not customary to do so, we give the proof in this introduction to show that it is just a few lines long and only requires the definition of bilateral Bailey pairs and elementary sum manipulations which are similar to the ones for the unilateral version.
\begin{proof}[Proof of Lemma~\ref{lem:key1}]
For all $n\in\mathbb{Z}$, we have
\begin{align*}
\sum_{j\leq n} \frac{\alpha'_j}{(q)_{n-j}(a)_{n+j}}&= \sum_{j\leq n} \frac{(1-a)}{(q)_{n-j}(a)_{n+j}}\left(\frac{\alpha_j}{1-aq^{2j}}-\frac{aq^{2j-2}\alpha_{j-1}}{1-aq^{2j-2}}\right)\\
&= \sum_{j\leq n} \frac{(1-a)\alpha_j}{(q)_{n-j}(a)_{n+j}\bigl(1-aq^{2j}\bigr)}-\sum_{j\leq n} \frac{(1-a)\bigl(1-q^{n-j}\bigr)aq^{2j}\alpha_{j}}{(q)_{n-j}(a)_{n+j+1}\bigl(1-aq^{2j}\bigr)}\\
&=\sum_{j\leq n} \frac{(1-a)\alpha_j}{(q)_{n-j}(a)_{n+j+1}\bigl(1-aq^{2j}\bigr)}\bigl(\bigl(1-aq^{n+j}\bigr)-aq^{2j}\bigl(1-q^{n-j}\bigr)\bigr)\\
&=\sum_{j\leq n} \frac{\alpha_j}{(q)_{n-j}(aq)_{n+j}}=\beta_n=\beta'_n,
\end{align*}
which is the desired result by~\eqref{bbp}.
\end{proof}
With this lemma, we can give an extremely simple proof of the bilateral Bailey lattice.
\begin{proof}[Proof of the bilateral Bailey lattice]
Start from a bilateral Bailey pair $(\alpha_n, \beta_n)$ relative to~$a$. Then apply Lemma~\ref{lem:key1} to obtain a bilateral Bailey pair $\bigl(\tilde{\alpha}_n, \tilde{\beta}_n\bigr)$ relative to $a/q$ and satisfying~\eqref{eq:baileylattice1}. Applying the bilateral Bailey lemma (see Theorem~\ref{thm:bilatbaileylemma}) to \smash{$\bigl(\tilde{\alpha}_n, \tilde{\beta}_n\bigr)$} with $a$ replaced by~$a/q$ gives the desired bilateral Bailey pair $(\alpha'_n, \beta'_n)$ relative to $a/q$.
\end{proof}

In addition to Lemma~\ref{lem:key1}, let us give a similarly simple lemma whose unilateral version is also due to McLaughlin in~\cite[Lemma~13.1\,(2)]{M} (see also Lovejoy \cite[Lemma 3.1]{Lo22}), and whose proof is very similar to the one of Lemma~\ref{lem:key1}.

\begin{Lemma}[key Lemma~2]\label{lem:key2}
If $(\alpha_n, \beta_n)$ is a bilateral Bailey pair relative to $a$, then $(\alpha'_n, \beta'_n)$ is a bilateral Bailey pair relative to $a/q$, where
\[%\label{eq:baileylattice2}
\alpha'_n=(1-a)\left(\frac{q^n\alpha_n}{1-aq^{2n}}-\frac{q^{n-1}\alpha_{n-1}}{1-aq^{2n-2}}\right), \qquad \beta'_n=q^n\beta_n,
\]
subject to convergence conditions on the sequences $\alpha_n$ and $\beta_n$, which make the relevant infinite series absolutely convergent.
\end{Lemma}

Using, as in the short proof of the bilateral Bailey lattice above, Lemma~\ref{lem:key2} followed by Theorem~\ref{thm:bilatbaileylemma} with $a$ replaced by $a/q$, we obtain the following new bilateral Bailey lattice, similar to Theorem~\ref{thm:bilatbaileylattice}. As far as we know, its unilateral version was also unknown until now.

\begin{Theorem}[new bilateral Bailey lattice]\label{thm:newbilatbaileylattice}
If $(\alpha_n, \beta_n)$ is a bilateral Bailey pair relative to~$a$, then $(\alpha'_n, \beta'_n)$ is a bilateral Bailey pair relative to $a/q$, where
\begin{equation*}\alpha'_n=\frac{(\rho,\sigma)_n(a/\rho\sigma)^n}{(a/\rho,a/\sigma)_n}(1-a)\left(\frac{q^n\alpha_n}{1-aq^{2n}}-\frac{q^{n-1}\alpha_{n-1}}{1-aq^{2n-2}}\right),
\end{equation*}
and
\begin{equation*}\beta'_n=\sum_{j\leq n}{(\rho,\sigma)_j(a/\rho\sigma)_{n-j}(a/\rho\sigma)^j\over (q)_{n-j}(a/\rho,a/\sigma)_n} q^j\beta_j,
\end{equation*}
subject to convergence conditions on the sequences $\alpha_n$ and $\beta_n$, which make the relevant infinite series absolutely convergent.
\end{Theorem}

Lemmas \ref{lem:key1} and \ref{lem:key2} can be generalised by adding an extra parameter $b$.
\begin{Lemma}[general lemma]\label{lem:baileylattices}
If $(\alpha_n, \beta_n)$ is a bilateral Bailey pair relative to $a$, then $(\alpha'_n, \beta'_n)$ is a bilateral Bailey pair relative to $a/q$, where
\[%\label{eq:baileylatticesalpha}
\alpha'_n=\frac{1-a}{1-b}\left(\frac{\bigl(1-bq^n\bigr) \alpha_n}{1-aq^{2n}}-\frac{q^{n-1}\bigl(aq^{n-1}-b\bigr) \alpha_{n-1}}{1-aq^{2n-2}}\right)
\qquad \text{and} \qquad \beta'_n=\frac{1-bq^n}{1-b}\beta_n.
\]
\end{Lemma}

\begin{Remark}
Lemma \ref{lem:key1} is the case $b=0$ and Lemma \ref{lem:key2} is the case $b \rightarrow \infty$ of Lemma \ref{lem:baileylattices}.
\end{Remark}

\begin{Remark}
While at first glance Lemma \ref{lem:baileylattices} seems more general than Lemmas \ref{lem:key1} and~\ref{lem:key2}, it is actually equivalent to these two lemmas taken together. Indeed, the bilateral Bailey pair in Lemma \ref{lem:baileylattices} is equal to $1/(1-b)$ times the bilateral Bailey pair of Lemma \ref{lem:key1} minus~${b/(1-b)}$ times the bilateral Bailey pair of Lemma \ref{lem:key2}. Using the fact that being a bilateral Bailey pair is stable under linear combination, Lemmas \ref{lem:key1} and \ref{lem:key2} imply Lemma \ref{lem:baileylattices}. Note that it is also possible to prove Lemma \ref{lem:baileylattices} directly with a similar method to the proof of Lemma \ref{lem:key1}.
\end{Remark}

Despite following from Lemmas \ref{lem:key1} and \ref{lem:key2}, this general Lemma \ref{lem:baileylattices} is still interesting as it provides in the unilateral case an ``inverse" to Lovejoy's Lemma 2.3 of \cite{Lo22}, which he first stated in \cite[equations (2.4) and (2.5)]{Lo04}. Actually, this result inspired our discovery of Lemma~\ref{lem:baileylattices}.

\begin{Lemma}[Lovejoy]\label{lem:love}
If $(\alpha_n, \beta_n)$ is a Bailey pair relative to $a$, then $(\alpha'_n, \beta'_n)$ is a Bailey pair relative to $aq$, where
\begin{equation*}
\alpha'_n= \frac{\bigl(1-aq^{2n+1}\bigr)(aq/b)_n(-b)^nq^{n(n-1)/2}}{(1-aq)(bq)_n} \sum_{r=0}^n \frac{(b)_r}{(aq/b)_r}(-b)^{-r}q^{-r(r-1)/2}\alpha_r,
\end{equation*}
and
\[
\beta'_n=\frac{1-b}{1-bq^n}\beta_n.
\]
\end{Lemma}

Moreover, from Lemma \ref{lem:baileylattices}, we deduce a very general theorem transforming bilateral Bailey pairs relative to $a$ into bilateral Bailey pairs relative to $aq^{-N}$. Recall the $M$-th elementary symmetric polynomial in $N$ variables defined for $0 \leq M \leq N$ as
\begin{equation*}e_M(X_1, \dots , X_N)= \sum_{1\leq i_1<i_2<\dots<i_M\leq N} X_{i_1} X_{i_2}\cdots X_{i_M},
\end{equation*}
and $e_M(X_1, \dots , X_N)=0$ if $M<0$ or $M>N$.

Recall also that for all $0\leq j \leq N$, the $q$-binomial coefficient is defined by
\begin{equation*}\left[{N\atop j}\right]_q=\left[{N\atop j}\right]:=\frac{(q)_N}{(q)_j(q)_{N-j}}.
\end{equation*}
We extend this definition to $j<0$ and $j>N$ by setting $\bigl[{N\atop j}\bigr]=0$, which, as $N\geq0$, is consistent with the definition of $q$-Pochhammer symbols with negative indices given above.
The general theorem can be stated as follows.

\begin{Theorem}[new bilateral Bailey lattice in higher dimension]\label{thm:multibaileylattice}
Let $(\alpha_n, \beta_n)$ be a bilateral Bailey pair relative to $a$. For all $N\geq 1$, define the pair \smash{$\bigl(\alpha^{(N)}_n, \beta^{(N)}_n\bigr)$} by
\begin{equation}
\alpha^{(N)}_n=\frac{\bigl(1-aq^{2n-N}\bigr)\bigl(aq^{1-N}\bigr)_N}{(1-b_1)\cdots(1-b_N)}\sum_{j \in \Z}(-1)^j\frac{q^{jn-j(j+1)/2}f_{N,j,n}(b_1,\ldots,b_N)}{\bigl(aq^{2n-N-j}\bigr)_{N+1}} \alpha_{n-j},\label{eq:multibaileylatticealpha}
\end{equation}
where
\begin{align}
f_{N,j,n}(b_1,\ldots,b_N) :={}& \sum_{M \in \Z} \sum_{u \in \Z} a^{u}q^{(M-j+u)(n-j+u)+u(n-N)} \begin{bmatrix}
M\\
j-u
\end{bmatrix} \begin{bmatrix}
N-M\\
u
\end{bmatrix}\nonumber\\
&\times (-1)^Me_M(b_1, \dots , b_N),\label{eq:pascalab}
\end{align}
and
\begin{equation}
\beta^{(N)}_n=\left(\prod_{i=1}^N \frac{1-b_iq^n}{1-b_i}\right)\beta_n.\label{eq:multibaileylatticebeta}
\end{equation}
Then, \smash{$\bigl(\alpha^{(N)}_n, \beta^{(N)}_n\bigr)$} is a bilateral Bailey pair relative to $aq^{-N}$.
\end{Theorem}

\begin{Remark}
When $j<0$ or $j > N$, we have $f_{N,j,n}(b_1,\ldots,b_N) =0$ because of the $q$-binomial coefficients in~\eqref{eq:pascalab}. Therefore, the sum in \eqref{eq:multibaileylatticealpha} is actually finite.
\end{Remark}

\begin{Remark}
Note that the sums over $M$ and $u$ in \eqref{eq:pascalab} could equivalently be taken from~$0$ to $N$ and from~$0$ to $j$, respectively. Indeed the $q$-binomial coefficients or $e_M(b_1,\dots ,b_N)$ naturally cancel outside of these ranges. However, the expression with infinite sums makes future calculations easier to write.
\end{Remark}

Again, Theorem \ref{thm:multibaileylattice} can be seen in the unilateral case as the inverse of a theorem of Lovejoy~\cite[Theorem~2.3]{Lo04}.

\begin{Theorem}[Lovejoy]\label{thm:multibaileylatticeLo}
Let $(\alpha_n, \beta_n)$ be a Bailey pair relative to $a$. For all $N\geq 1$, define the pair \smash{$\bigl(\alpha^{(N)}_n, \beta^{(N)}_n\bigr)$} by
\begin{align*}
\alpha^{(N)}_n={}&\frac{\bigl(1-aq^{2n+N}\bigr)\bigl(aq^{N}/b_N\bigr)_n(-b_N)^nq^{\binom{n}{2}}}{\bigl(1-aq^N\bigr)(b_Nq)_n}
\\
&\times \sum_{n \geq n_N \geq \cdots \geq n_1 \geq 0} \frac{\bigl(1-aq^{2n_2+1}\bigr) \cdots \bigl(1-aq^{2n_N+N-1}\bigr)(aq/b_1)_{n_2} \cdots \bigl(aq^{N-1}/b_{N-1}\bigr)_{n_N}}{(1-aq) \cdots \bigl(1-aq^{N-1}\bigr)(aq/b_1)_{n_1} \cdots \bigl(aq^{N}/b_{N}\bigr)_{n_N}}
\\
&\times \frac{(b_1)_{n_1} \cdots (b_N)_{n_N}}{(b_1q)_{n_2} \cdots (b_{N-1}q)_{n_N}} b_1^{n_2-n_1} \cdots b_{N-1}^{n_N-n_{N-1}} b_N^{-n_N} (-1)^{n_1} q^{-\binom{n_1}{2}} \alpha_{n_1},
\end{align*}
and
\[
\beta^{(N)}_n=\left(\prod_{i=1}^N \frac{1-b_i}{1-b_iq^n}\right)\beta_n.
\]
Then, \smash{$\bigl(\alpha^{(N)}_n, \beta^{(N)}_n\bigr)$} is a Bailey pair relative to $aq^{N}$.
\end{Theorem}

As particular cases of Theorem \ref{thm:multibaileylattice}, we recover and generalise to the bilateral case some Bailey lattices due to Warnaar, as well as discover new simple ones (see Section \ref{sec:newBailey}).

Moreover, we take advantage of the bilateral aspect of our results by using a bilateral Bailey pair (in the special case $a=q^m$, see \eqref{ambubp}) instead of the classical unit Bailey pair, and obtain new generalisations, which we call $m$-versions, of the Andrews--Gordon identities, the Bressoud identities, and new companions to Bressoud's identities which we very recently discovered combinatorially \cite{DJK} (see \eqref{fij0}--\eqref{fij}).
The $m$-version of the Andrews--Gordon identities is as follows.
\begin{Theorem}[$m$-version of the Andrews--Gordon identities]\label{thm:mag}
Let $m\geq 0$, $r \geq 2$, and $0 \leq i \leq r$ be three integers. We have
\begin{gather}
\sum_{s_1\geq\dots\geq s_{r}\geq-\left\lfloor m/2\right\rfloor}\frac{q^{s_1^2+\dots+s_{r}^2+m(s_1+\dots+s_r)-s_1-\dots-s_i}}{(q)_{s_1-s_2}\cdots(q)_{s_{r-1}-s_r}}(-1)^{s_r}q^{\binom{s_r}{2}}\left[{m+s_r\atop m+2s_r}\right]\nonumber\\
\qquad=\sum_{k=0}^{i}q^{mk}\frac{\bigl(q^{2r+1},q^{(m+1)r-i+2k},q^{(1-m)r+i-2k+1};q^{2r+1}\bigr)_\infty}{(q)_\infty}.\label{mag}
\end{gather}
\end{Theorem}
Note that in~\cite[equation~(3.21)]{BP}, Berkovich and Paule prove a different $m$-version of the Andrews--Gordon identities, in which $m$ is negative.

The $m$-version of our new companions to Bressoud's identities is the following.
\begin{Theorem}[$m$-version of our identities]\label{thm:mfij}
Let $m\geq 0$, $r \geq 2$, and $0 \leq i \leq r$ be integers. Then
\begin{gather}
\sum_{s_1\geq\dots\geq s_{r}\geq-\left\lfloor m/2\right\rfloor}\frac{q^{s_1^2+\dots+s_{r}^2+m(s_1+\dots+s_{r-1})-s_1-\dots-s_{i}+s_{r-1}-2s_r}(-q)_{m+2s_r}}{(q)_{s_1-s_2}\cdots(q)_{s_{r-2}-s_{r-1}}\bigl(q^2;q^2\bigr)_{s_{r-1}-s_{r}}}(-1)^{s_r}\left[{m+s_r\atop m+2s_r}\right]_{q^2}\nonumber\\
\qquad=\sum_{k=0}^{i}q^{mk}\frac{\bigl(q^{2r},q^{(m+1)(r-1)-i+2k},q^{(1-m)r+m+i-2k+1};q^{2r}\bigr)_\infty}{(q)_\infty}.\label{mfij}
\end{gather}
\end{Theorem}

The $m$-version of the classical Bressoud identities (namely the even moduli counterpart of the Andrews--Gordon identities) is a bit less elegant (see Theorem~\ref{thm:mb}). Recall that in~\cite[p.~387]{W03}, there is a different $m$-version of the even moduli case.

Other famous identities of the Rogers--Ramanujan type, which where found by G\"ollnitz~\cite{Gol} and Gordon~\cite{Go65} independently, can be stated as follows.
\begin{Theorem}[G\"ollnitz--Gordon identities]
 \label{th:GG}
 Let $i=0$ or $1$. Then
\begin{equation*}
 \sum_{n \geq 0} \frac{q^{n^2+ 2(1-i)n}\bigl(-q;q^2\bigr)_n}{\bigl(q^2;q^2\bigr)_n} = \frac{1}{\bigl(q^{3-2i},q^4,q^{5+2i};q^8\bigr)_{\infty}}.
\end{equation*}
\end{Theorem}

As for the Rogers--Ramanujan identities, there are combinatorial interpretations and multisum generalisations of the G\"ollnitz--Gordon identities in the spirit of the Andrews--Gordon identities~\eqref{AG} (see, for instance, the recent paper~\cite{HZ}). Actually, Bressoud proved in~\cite{Br80} three different such generalisations, which are listed as (3.6)--(3.8) in his paper (he also proved another formula of the same kind, namely~\cite[equation (3.9)]{Br80}, which is so similar to~\cite[equation (3.8)]{Br80} that it is considered in~\cite{HZ} as a generalisation of the G\"ollnitz--Gordon identities, although it is not \emph{stricto sensu} the case).

While looking for $m$-versions of all these Bressoud--G\"ollnitz--Gordon identities, we discovered the following result, which surprisingly interpolates between the classical Bressoud identities and~\cite[equation (3.6)]{Br80}.

\begin{Theorem}[$m$-version of the Bressoud and Bressoud--G\"ollnitz--Gordon identities]\label{thm:mbr-3.6}
Let ${m\geq 0}$, $r \geq 2$, and $0 \leq i \leq r$ be integers. Then
\begin{gather}
\sum_{s_1\geq\dots\geq s_{r}\geq-\left\lfloor m/2\right\rfloor}\frac{q^{s_1^2+\dots+s_{r}^2+m(s_1+\dots+s_{r})-s_1-\dots-s_{i}-(m+1)s_r/2}\bigl(-q^{(m+1)/2}\bigr)_{s_{r}}}{(q)_{s_1-s_2}\cdots(q)_{s_{r-1}-s_{r}}\bigl(-q^{(m+1)/2}\bigr)_{s_{r-1}}}(-1)^{s_r}\left[{m+s_r\atop m+2s_r}\right]\nonumber\\
\qquad=\sum_{k=0}^{i}q^{mk}\frac{\bigl(q^{2r},q^{(m+1)r-i+2k-(m+1)/2},q^{(1-m)r+i-2k+(m+1)/2};q^{2r}\bigr)_\infty}{(q)_\infty}.\label{mbr-3.6}
\end{gather}
\end{Theorem}

Actually, in~\cite{Br80}, Bressoud proved a very general multi-parameter identity (see Theorem~\ref{thm:Bressoud} below), of which the cases $m=0$ and $m=1$ of Theorems~\ref{thm:mag},~\ref{thm:mb}, and~\ref{thm:mbr-3.6} are particular cases. This led us to believe that Theorem~\ref{thm:Bressoud} could be proved using the classical Bailey lattice (see Theorem~\ref{thm:baileylattice}), but we did not succeed. However we managed to prove it in a simple way by using the unilateral version of our new Bailey lattice (see Theorem~\ref{thm:newbilatbaileylattice}), see Section~\ref{sec:bressoud}.
Moreover, the cases $m=0$ and $m=1$ of Theorem~\ref{thm:mfij} do not seem to follow from Theorem~\ref{thm:Bressoud}. So the Bailey lattice approach appears to be more general.

The paper is organised as follows. In Section~\ref{sec:mAG}, we use our bilateral Bailey lattices to prove general results and deduce $m$-versions of many classical identities, among which Theorems~\ref{thm:mag}--\ref{thm:mbr-3.6}.
In Section~\ref{sec:newBailey}, we give $N$-iterations of our Bailey lattices, generalise some $N$-Bailey lattices of Warnaar to the bilateral case, and prove Theorem~\ref{thm:multibaileylattice}.
In Section~\ref{sec:bressoud}, we show how to derive a new proof of Bressoud's theorem with our new Bailey lattice of Theorem~\ref{thm:newbilatbaileylattice}, and why we fail when trying to do the same using the classical Bailey lattice. Finally, we conclude with a~short section listing some open problems.

%%%%%%%%%%%%%%%%%%%%%%%%%%%%%%%%%%%%%%%%%%%%%%%%%%%%%%%%%%%%%%%%%%%%%%%%%%%
\section[New m-versions of the Andrews--Gordon identities and others]{New $\boldsymbol{m}$-versions of the Andrews--Gordon identities and others}\label{sec:mAG}
%%%%%%%%%%%%%%%%%%%%%%%%%%%%%%%%%%%%%%%%%%%%%%%%%%%%%%%%%%%%%%%%%%%%%%%%%%%

%%%%%%%%%%%%%%%%%
\subsection{Combining bilateral Bailey lemmas and lattices}
%%%%%%%%%%%%%%%%%

In~\cite{AAB}, many applications of the Bailey lattice (see Theorem~\ref{thm:baileylattice}) are provided, among which a~general result, obtained in~\cite[Theorem~3.1]{AAB} by iterating $r-i$ times Theorem~\ref{thm:baileylemma}, then using Theorem~\ref{thm:baileylattice}, and finally $i-1$ times Theorem~\ref{thm:baileylemma} with $a$ replaced by $a/q$. Using the same process in our bilateral point of view, replacing Theorem~\ref{thm:baileylemma} (resp.\ Theorem~\ref{thm:baileylattice}) by Theorem~\ref{thm:bilatbaileylemma} (resp.\ Theorem~\ref{thm:bilatbaileylattice}), we derive the following generalisation of~\cite[Theorem~3.1]{AAB}.

\begin{Theorem}\label{thm:bilatncsqbaileylattice}
If $(\alpha_n, \beta_n)$ is a bilateral Bailey pair relative to $a$, then for all integers $0\leq i\leq r$ and $n\in \Z$, we have
\begin{gather*}%\label{ncsqlattice}
\sum_{n\geq s_1\geq\dots\geq s_{r}}\frac{a^{s_1+\dots+s_r}q^{s_{i+1}+\dots+s_{r}}\beta_{s_r}}{(\rho_1\sigma_1)^{s_1}\cdots(\rho_r\sigma_r)^{s_r}}\frac{(\rho_1,\sigma_1)_{s_1} \cdots(\rho_r,\sigma_r)_{s_r}}{(q)_{n-s_1}(q)_{s_1-s_2}\cdots(q)_{s_{r-1}-s_r}}\\
\qquad{}\times\frac{(a/\rho_1\sigma_1)_{n-s_1}(a/\rho_2\sigma_2)_{s_1-s_2}\cdots(a/\rho_i\sigma_i)_{s_{i-1}-s_i}}{(a/\rho_1,a/\sigma_1)_{n} (a/\rho_2,a/\sigma_2)_{s_1}\cdots(a/\rho_i,a/\sigma_i)_{s_{i-1}}}\\
\qquad{}\times\frac{(aq/\rho_{i+1}\sigma_{i+1})_{s_i-s_{i+1}}\cdots(aq/\rho_r\sigma_r)_{s_{r-1}-s_r}}{(aq/\rho_{i+1},aq/\sigma_{i+1})_{s_i} \cdots(aq/\rho_r,aq/\sigma_r)_{s_{r-1}}}\\
\phantom{\qquad{}\times}{}=\sum_{j\leq n}
\frac{(\rho_1,\sigma_1,\dots,\rho_i,\sigma_i)_j(\rho_1\sigma_1\cdots\rho_i\sigma_i)^{-j}a^{ij}(1-a)}{(q)_{n-j}(a)_{n+j} (a/\rho_1,a/\sigma_1,\dots,a/\rho_i,a/\sigma_i)_j}\\
\phantom{\qquad{}\times=}{}\times\left(\frac{(\rho_{i+1},\sigma_{i+1},\dots,\rho_r,\sigma_r)_j(\rho_{i+1} \sigma_{i+1}\cdots\rho_r\sigma_r)^{-j}(aq)^{(r-i)j}\alpha_j}{(aq/\rho_{i+1},aq/\sigma_{i+1},\dots,aq/\rho_r,aq/\sigma_r)_j\bigl(1-aq^{2j}\bigr)}\right.\\
\phantom{\qquad{}\times=}{}\left.- \frac{(\rho_{i+1},\sigma_{i+1},\dots,\rho_r,\sigma_r)_{j-1}(\rho_{i+1}\sigma_{i+1}\cdots\rho_r\sigma_r)^{-j+1}(aq)^{(r-i)(j-1)}aq^{2j-2} \alpha_{j-1}}{(aq/\rho_{i+1},aq/\sigma_{i+1},\dots,aq/\rho_r,aq/\sigma_r)_{j-1}\bigl(1-aq^{2j-2}\bigr)}\right),
\end{gather*}
subject to convergence conditions on the sequences $\alpha_n$ and $\beta_n$, which make the relevant infinite series absolutely convergent.
\end{Theorem}

In this section, we will consider the special case below where all parameters $\rho_j,\sigma_j\to\infty$ and at the end $n\to+\infty$, which is a bilateral generalisation of~\cite[Corollary 4.2]{AAB}. (We also shifted the index $j$ to $j+1$ in the terms involving $\alpha_{j-1}$.)

\begin{Corollary}\label{coro:bilatbaileylattice}
If $(\alpha_n, \beta_n)$ is a bilateral Bailey pair relative to $a$, then for all integers $0\leq i\leq r$, we have
\begin{equation}\label{corobilatlattice}
\sum_{s_1\geq\dots\geq s_{r}}\frac{a^{s_1+\dots+s_r}q^{s_1^2+\dots+s_{r}^2-s_1-\dots-s_{i}}}{(q)_{s_1-s_2}\cdots(q)_{s_{r-1}-s_r}}\beta_{s_r} =\frac{1}{(aq)_\infty}\sum_{j\in\mathbb{Z}}a^{rj}q^{rj^2-ij}\frac{1-a^{i+1}q^{2j(i+1)}}{1-aq^{2j}}\alpha_j,
\end{equation}
subject to convergence conditions on the sequences $\alpha_n$ and $\beta_n$, which make the relevant infinite series absolutely convergent.
\end{Corollary}

In~\cite{AAB}, Agarwal, Andrews and Bressoud prove the Andrews--Gordon identities~\eqref{AG} in the following way. They apply Corollary~\ref{coro:bilatbaileylattice} to the unit Bailey pair~\eqref{ubp} (which we recall is unilateral) with $a=q$, factorise the right-hand side using the Jacobi triple product identity~\cite[formula~(II.28)]{GR}
\begin{equation}\label{jtp}
\sum_{j\in\mathbb{Z}}(-1)^jz^jq^{j(j-1)/2}=(q,z,q/z;q)_\infty,
\end{equation}
and replace $i$ by $i-1$.

Regarding $m$-versions of Bressoud and Bressoud--G\"ollnitz--Gordon type identities, we will also need the more general case below where all parameters except $\rho_1$, $\rho_r$ tend to $\infty$ ($\rho_1$, $\rho_r$ are replaced by $b$, $c$ below).

\begin{Corollary}\label{coro:bilatbaileylattice2}
If $(\alpha_n, \beta_n)$ is a bilateral Bailey pair relative to $a$, then for all integers $0\leq i\leq r$, we have
\begin{gather}
\sum_{s_1\geq\dots\geq s_{r}}\frac{a^{s_1+\dots+s_r}q^{s_1^2/2+s_2^2+\dots+s_{r-1}^2+s_{r}^2/2-s_1/2-s_2-\dots-s_{i}+s_r/2}}{(q)_{s_1-s_2}\cdots(q)_{s_{r-1}-s_r}}(-1)^{s_1+s_r}\frac{(b)_{s_1}(c)_{s_r}}{b^{s_1}c^{s_r}(aq/c)_{s_{r-1}}}\beta_{s_r}\nonumber\\
\qquad=\frac{(a/b)_\infty}{(aq)_\infty}\sum_{j\in\mathbb{Z}}\frac{a^{rj}q^{(r-1)j^2-ij+j}}{1-aq^{2j}}\frac{(b,c)_j}{b^jc^j(a/b,aq/c)_j}
\left(1+a^{i+1}q^{j(2i+1)}\frac{1-bq^j}{b-aq^j}\right)\alpha_j,\label{coro2bilatlattice}
\end{gather}
subject to convergence conditions on the sequences $\alpha_n$ and $\beta_n$, which make the relevant infinite series absolutely convergent.
\end{Corollary}

Of course when $b,c\to\infty$ in~\eqref{coro2bilatlattice}, one gets~\eqref{corobilatlattice}

\begin{Remark}\label{rk:coronewbilatbaileylattice}
One can obtain a result similar to Theorem~\ref{thm:bilatncsqbaileylattice} by using Theorem~\ref{thm:newbilatbaileylattice} instead of Theorem~\ref{thm:bilatbaileylattice}. However,
since the limiting case of interest is nothing but Corollary~\ref{coro:bilatbaileylattice2} with $i$ replaced by $i-1$, we have decided to omit this additional theorem.
\end{Remark}

%%%%%%%%%%%%%%%%%
\subsection{Bilateral Bailey pairs}
%%%%%%%%%%%%%%%%%

In~\cite{J}, the bilateral Bailey lemma given in Theorem~\ref{thm:bilatbaileylemma} is studied in particular by considering the case where $a=q^m$ for a non-negative integer $m$ (this instance is called shifted Bailey lemma in~\cite{J}). The following bilateral (actually shifted) Bailey pair, which was already mentioned in another form in~\cite{ASW}, is considered
\begin{equation}\label{mbubp}
\alpha_n=(-1)^nq^{\binom{n}{2}}\qquad \mbox{and}\qquad\beta_n=(q)_m(-1)^nq^{\binom{n}{2}}\left[{m+n\atop m+2n}\right].
\end{equation}

Taking $m=0$ and $m=1$ in~\eqref{mbubp} yields Bailey pairs equivalent to the cases $a=1$ and~${a=q}$ of the unit Bailey pair~\eqref{ubp}. Note that choosing $\beta_n=\delta_{n,0}$ and computing $\alpha_n$ by the inversion~\eqref{bilatinversion} would not provide a new bilateral Bailey pair, as can be seen by Remark~\ref{rk:bilatinversion}: it returns the usual unit Bailey pair~\eqref{ubp}. However, to use in full generality the bilateral point of view while keeping $a$ general, it would be natural to consider
\begin{equation}\label{ambubp}
\alpha_n=(-1)^{n+m}q^{\binom{n+m}{2}}\frac{1-aq^{2n}}{1-a}\frac{(a)_{n-m}}{(q)_{n+m}}\qquad \mbox{and}\qquad\beta_n=\delta_{n,-m},
\end{equation}
where we made use of the inversion~\eqref{bilatinversion}. However, applying Corollary~\ref{coro:bilatbaileylattice} to the bilateral Bailey pair~\eqref{ambubp} does not provide any interesting generalisation (like the $m$-versions of the next section in the case of \eqref{mbubp}) of~\eqref{AG}, but a formula which is equivalent to~\eqref{AG} for all $m$.
 Indeed, by applying~\eqref{corobilatlattice} to~\eqref{ambubp} and replacing the index $j$ by $j-m$, a few classical $q$-series manipulations show that at the end there is no genuine dependence on $m$.

%%%%%%%%%%%%%%%%%
\subsection[m-versions of the Andrews--Gordon identities]{$\boldsymbol{m}$-versions of the Andrews--Gordon identities}
%%%%%%%%%%%%%%%%%

Recall that the Andrews--Gordon identities~\eqref{AG} arise in~\cite{Br80} in pair with a similar formula~\cite[equation~(3.3)]{Br80}, valid for all integers $r \geq 2$ and $0\leq i\leq r-1$
\begin{equation}\label{Br3.3}
\sum_{s_1\geq\dots\geq s_{r-1}\geq0}\frac{q^{s_1^2+\dots+s_{r-1}^2-s_1-\dots-s_i}}{(q)_{s_1-s_2}\cdots(q)_{s_{r-2}-s_{r-1}}(q)_{s_{r-1}}} =\sum_{k=0}^{i}\frac{\bigl(q^{2r+1},q^{r-i+k},q^{r+i-k+1};q^{2r+1}\bigr)_\infty}{(q)_\infty}.
\end{equation}
Note that there is a small mistake in Bressoud's paper: in his formula~\cite[equation~(3.3)]{Br80}, $\pm(k-r+i)$ (in his notation) has to be changed to $\pm(k-r+i+1)$. Identity~\eqref{Br3.3} is explained combinatorially in~\cite{DJK}, while it is used in~\cite{ADJM} to solve a combinatorial conjecture of Afsharijoo arising from commutative algebra.

We show that~\eqref{AG} and~\eqref{Br3.3} can be embedded in a single formula involving the integer~$m$ from the previous subsection: this is Theorem \ref{thm:mag}, the $m$-version of the Andrews--Gordon identities. Our proof relies on Corollary \ref{coro:bilatbaileylattice}, which itself is a consequence of our bilateral Bailey lattice.

\begin{proof}[Proof of Theorem \ref{thm:mag}]
Apply Corollary~\ref{coro:bilatbaileylattice} to the bilateral Bailey pair~\eqref{mbubp} with $a=q^m$ and divide both sides by $(q)_m$, this yields the desired left-hand side of~\eqref{mag}. Regarding the right-hand side, one gets
\begin{equation*}\frac{1}{(q)_\infty}\sum_{j\in\mathbb{Z}}q^{rj^2-ij+mrj}\frac{1-q^{(m+2j)(i+1)}}{1-q^{m+2j}}(-1)^{j}q^{\binom{j}{2}},
\end{equation*}
which by expanding the denominator in a geometric series yields
\begin{gather*}
\frac{1}{(q)_\infty}\sum_{k=0}^iq^{mk}\sum_{j\in\mathbb{Z}}q^{rj^2-ij+mrj+2kj}(-1)^{j}q^{\binom{j}{2}}\\
\qquad=\frac{1}{(q)_\infty}\sum_{k=0}^iq^{mk}\sum_{j\in\mathbb{Z}}(-1)^{j}q^{(2r+1)\binom{j}{2}}q^{j((m+1)r-i+2k)}.
\end{gather*}
This gives the result by using the Jacobi triple product identity~\eqref{jtp}.
\end{proof}

The case $i=0$ of Theorem \ref{thm:mag} is \cite[Theorem~2.3, equation~(2.3)]{J}, where specialisations of this formula are also studied further. Taking $m=0$ in~\eqref{mag} forces the index $s_r$ to be $0$, therefore the left-hand side is the one of~\eqref{Br3.3}. The right-hand sides actually also coincide: it is obvious for the even indices $2k$ on the right-hand side of~\eqref{Br3.3} (for $0\leq 2k\leq i$), while the odd indices~${2k+1}$ correspond to indices $i-k$ on the right-hand side of~\eqref{mag}. Taking~${m=1}$ in~\eqref{mag} also yields~$s_r$ to be $0$, therefore the left-hand side is the one of~\eqref{AG} (in which $i$ is replaced by~${i+1}$). Regarding the right-hand sides, the one of~\eqref{AG} is given by the first term~${k=0}$ in~\eqref{mag} (with $i$ replaced by $i-1$), while the sum from $1$ to $i$ actually cancels, even though it is not immediate at first sight.

\subsection[m-versions of Bressoud's even moduli counterparts]{$\boldsymbol{m}$-versions of Bressoud's even moduli counterparts}

In~\cite{B}, Bressoud found the counterpart for even moduli to the Andrews--Gordon identities~\eqref{AG}
\begin{gather}
\sum_{s_1\geq\dots\geq s_{r-1}\geq0}\frac{q^{s_1^2+\dots+s_{r-1}^2+s_{i}+\dots+s_{r-1}}}{(q)_{s_1-s_2}\cdots(q)_{s_{r-2}-s_{r-1}}\bigl(q^2;q^2\bigr)_{s_{r-1}}}
=\frac{\bigl(q^{2r},q^{i},q^{2r-i};q^{2r}\bigr)_\infty}{(q)_\infty},\label{B}
\end{gather}
where $r \geq 2$ and $1\leq i\leq r$ are fixed integers. As for the Andrews--Gordon identities, there is a counterpart for~\eqref{B} similar to~\eqref{Br3.3} which is proved in~\cite[equation~(3.5)]{Br80} and explained combinatorially in~\cite{DJK}
\begin{equation}\label{Br3.5}
\sum_{s_1\geq\dots\geq s_{r-1}\geq0}\frac{q^{s_1^2+\dots+s_{r-1}^2-s_1-\dots-s_i}}{(q)_{s_1-s_2}\cdots(q)_{s_{r-2}-s_{r-1}}\bigl(q^2;q^2\bigr)_{s_{r-1}}} =\sum_{k=0}^{i}\frac{\bigl(q^{2r},q^{r-i+2k},q^{r+i-2k};q^{2r}\bigr)_\infty}{(q)_\infty},
\end{equation}
for all integers $r \geq 2$ and $0\leq i\leq r-1$.

In this subsection, we aim to find a generalisation of both formulas above, in the spirit of Theorem~\ref{thm:mag}. To do so, we will need the following bilateral version of~\cite[Theorem 2.5]{BIS}, which changes the basis $q$ to $q^2$.
\begin{Theorem}
If $(\alpha_n, \beta_n)$ is a bilateral Bailey pair relative to $a$, then so is $(\alpha'_n, \beta'_n)$, where
\begin{equation*}\alpha'_n=\frac{(-b)_n}{(-aq/b)_n}\frac{1+a}{1+aq^{2n}}b^{-n}q^{n-\binom{n}{2}}\alpha_n\bigl(a^2,q^2\bigr)
\end{equation*}
and
\begin{equation*}\beta'_n=\sum_{j\leq n}\frac{(-a)_{2j}\bigl(b^2;q^2\bigr)_j\bigl(q^{-j+1}/b,bq^{j}\bigr)_{n-j}}{(b,-aq/b)_n\bigl(q^2;q^2\bigr)_{n-j}}b^{-j}q^{j-\binom{j}{2}}\beta_j\bigl(a^2,q^2\bigr),
\end{equation*}
provided the relevant series are absolutely convergent. Here $\alpha_n\bigl(a^2,q^2\bigr)$ and $\beta_n\bigl(a^2,q^2\bigr)$ means that $a$ and $q$ are replaced by $a^2$ and $q^2$ in the bilateral Bailey pair.
\end{Theorem}

\begin{proof}
As in~\cite{BIS}, we only need to use the definition \eqref{bbp} of a bilateral Bailey pair, interchange summations and apply \cite[formula~(2.2)]{BIS}.
\end{proof}
As a consequence, letting $b\to+\infty$, we derive the following bilateral Bailey pair, therefore generalising (D4) in~\cite{BIS}
\begin{equation}\label{bilatD4}
\alpha'_n=\frac{1+a}{1+aq^{2n}}q^n\alpha_n\bigl(a^2,q^2\bigr)\qquad\mbox{and}\qquad\beta'_n=\sum_{j\leq n}\frac{(-a)_{2j}}{\bigl(q^2;q^2\bigr)_{n-j}}q^{j}\beta_j\bigl(a^2,q^2\bigr).
\end{equation}

Note that there are many other changes-of-base results in the literature complementing those of~\cite{BIS}. All of them should allow for the same bilateralisation as well. Nevertheless we only consider here the ones that we need for our purposes (see also~\eqref{bilatD1}).

Now we are ready to give our result.

\begin{Theorem}[$m$-version of the Bressoud identities]\label{thm:mb}
Let $m\geq 0$, $r \geq 2$, and $0 \leq i \leq r$ be three integers. We have
\begin{gather}
\sum_{s_1\geq\dots\geq s_{r}\geq-\left\lfloor m/2\right\rfloor}\!\!\!\!\!\frac{q^{s_1^2+\dots+s_{r}^2+m(s_1+\dots+s_{r-1})-s_1-\dots-s_{i}}({-}q)_{m+2s_r-1}}{(q)_{s_1-s_2}\cdots(q)_{s_{r-2} -s_{r-1}}\bigl(q^2;q^2\bigr)_{s_{r-1}-s_{r}}}(-1)^{s_r}\!\!\left[{m+s_r\atop m+2s_r}\right]_{q^2}\! =a_m,\!\!\!\!\label{mb}
\end{gather}
where
\begin{equation*}a_{2m}=\sum_{k=0}^{i}\sum_{\ell=0}^{2m}(-1)^\ell q^{2mk+2m\ell}\frac{\bigl(q^{2r},q^{2m(r-1)+r-i+2k+2\ell},q^{r+i-2m(r-1)-2k-2\ell};q^{2r}\bigr)_\infty}{2(q)_\infty},\end{equation*}
and
\begin{equation*}a_{2m+1}=(-1)^mq^{(2-r)m^2+(1+i-r)m}\sum_{\ell=0}^{m}q^{2\ell}\frac{\bigl(q^{2r},q^{2r-2m-1-i+4\ell},q^{i+2m+1-4\ell};q^{2r}\bigr)_\infty}{(q)_\infty}.
\end{equation*}
\end{Theorem}

\begin{proof}
We start from the bilateral Bailey pair~\eqref{mbubp} with $a=q^m$, to which we apply~\eqref{bilatD4}. This results in the bilateral Bailey pair
\begin{equation*}
\alpha_n=(-1)^nq^{n^2}\frac{1+q^{m}}{1+q^{m+2n}},\qquad\beta_n=\bigl(q^2;q^2\bigr)_m\sum_{j\leq n}(-1)^jq^{j^2}\frac{(-q^m)_{2j}}{\bigl(q^2;q^2\bigr)_{n-j}}\left[{m+j\atop m+2j}\right]_{q^2}.
\end{equation*}
Then apply Corollary~\ref{coro:bilatbaileylattice} with $a=q^m$ and $r$ replaced by $r-1$ to the above bilateral Bailey pair and divide both sides by $(1+q^m)(q)_m$: the left-hand side is the desired one ($\beta_n$ above is $\beta_{s_{r-1}}$ while $j=s_r$). The right-hand side is equal to
\begin{equation}\label{am1}
a_m=\frac{1}{(q)_\infty}\sum_{j\in\mathbb{Z}}(-1)^jq^{rj^2-ij+m(r-1)j}\frac{1-q^{(m+2j)(i+1)}}{1-q^{m+2j}}\frac{1}{1+q^{m+2j}}.
\end{equation}
Shifting the index of summation $j$ to $-j-m$ yields after rearranging
\begin{equation}\label{am2}
a_m=\frac{1}{(q)_\infty}\sum_{j\in\mathbb{Z}}(-1)^jq^{rj^2-ij+m(r-1)j}\frac{1-q^{(m+2j)(i+1)}}{1-q^{m+2j}}\frac{(-1)^mq^{(m+2j)(m+1)}}{1+q^{m+2j}}.
\end{equation}
Therefore, we get by adding~\eqref{am1} and~\eqref{am2}
\begin{equation*}
a_{2m}=\frac{1}{2(q)_\infty}\sum_{j\in\mathbb{Z}}(-1)^jq^{rj^2-ij+2m(r-1)j}\frac{1-q^{(2m+2j)(i+1)}}{1-q^{2m+2j}}\frac{1+q^{(2m+2j)(2m+1)}}{1+q^{2m+2j}},
\end{equation*}
in which we can expand both denominators in geometric series and obtain the desired result by using~\eqref{jtp}. Summing~\eqref{am1} and~\eqref{am2} gives in the odd case
\begin{equation*}
a_{2m+1}=\frac{1}{2(q)_\infty}\sum_{j\in\mathbb{Z}}(-1)^jq^{rj^2-ij+(2m+1)(r-1)j}\bigl(1-q^{(2m+1+2j)(i+1)}\bigr)\frac{1-q^{(2m+1+2j)(2m+2)}}{1-q^{2(2m+1+2j)}}.
\end{equation*}
Expanding the denominator in a geometric series and using~\eqref{jtp} yields
\begin{align*}
a_{2m+1}={}&\frac{1}{2(q)_\infty}\left(\sum_{\ell=0}^{m}q^{(4m+2)\ell}\bigl(q^{2r},q^{(2m+1)(r-1)+r-i+4\ell},q^{i-(2m+1)(r-1)+r-4\ell};q^{2r}\bigr)_\infty\right.\\
&\left.-\sum_{\ell=0}^{m}q^{(2m+1)(i+1)+(4m+2)\ell}\right.\\
&\left.\times\bigl(q^{2r},q^{(2m+1)(r-1)+r+i+2+4\ell},q^{-i-2-(2m+1)(r-1)+r-4\ell};q^{2r}\bigr)_\infty\right).
\end{align*}
Then observe that we can remove the $2mr$ factors in the infinite products by using in the first sum
\begin{gather*}
q^{(4m+2)\ell}\bigl(q^{(2m+1)(r-1)+r-i+4\ell},q^{i-(2m+1)(r-1)+r-4\ell};q^{2r}\bigr)_\infty\\
\qquad=(-1)^mq^{(2-r)m^2+(1+i-r)m+2\ell}\bigl(q^{2r-2m-1-i+4\ell},q^{i+2m+1-4\ell};q^{2r}\bigr)_\infty,
\end{gather*}
and in the second one
\begin{gather*}
q^{(2m+1)(i+1)+(4m+2)\ell}\bigl(q^{(2m+1)(r-1)+r+i+2+4\ell},q^{-i-2-(2m+1)(r-1)+r-4\ell};q^{2r}\bigr)_\infty\\
\qquad=(-1)^{m+1}q^{(2-r)m^2+(1+i-r)m-2\ell+2m}\bigl(q^{-2m+1+i+4\ell},q^{2r+2m-i-1-4\ell};q^{2r}\bigr)_\infty.
\end{gather*}
Therefore, replacing $\ell$ by $m-\ell$ in the second sum in $a_{2m+1}$ above yields the result.
\end{proof}

 Taking $m=0$ in~\eqref{mb} forces the index $s_r$ to be $0$, therefore we obtain the identity~\eqref{Br3.5} multiplied by $1/2$. Taking $m=1$ in~\eqref{mb} also yields $s_r$ to be $0$, therefore we get~\eqref{B} in which~$i$ is replaced by $i+1$.

\begin{Remark}\label{rk:iodd}
Contrary to Theorem~\ref{thm:mag}, we had to consider the parity of $m$ to use the Jacobi triple product~\eqref{jtp} in~\eqref{am1}. However, we managed to find a general expression for $a_m$, but only when $i$ is odd, writing
\begin{equation*}a_m=\frac{1}{(q)_\infty}\sum_{j\in\mathbb{Z}}(-1)^jq^{rj^2-ij+m(r-1)j}\frac{1-q^{(2m+4j)(i+1)/2}}{1-q^{2m+4j}}.
\end{equation*}
This gives for odd $i$ and any non-negative integer $m$
\begin{equation*}a_m=\sum_{k=0}^{(i-1)/2}q^{2mk}\frac{\bigl(q^{2r},q^{m(r-1)+r-i+4k},q^{r+i-m(r-1)-4k};q^{2r}\bigr)_\infty}{(q)_\infty}.
\end{equation*}
\end{Remark}

%%%%%%%%%%%%%%%%%
\subsection[m-versions of new even moduli counterparts]{$\boldsymbol{m}$-versions of new even moduli counterparts}
%%%%%%%%%%%%%%%%%
In~\cite{DJK}, while studying combinatorial interpretations of the Andrews--Gordon and Bressoud identities (\eqref{AG},~\eqref{Br3.3} and~\eqref{B}--\eqref{Br3.5}), the authors discovered in a purely combinatorial way the following pair of formulas, to be compared with~\eqref{B} and~\eqref{Br3.5}
\begin{gather}
(1+q)\sum_{s_1\geq\dots\geq s_{r-1}\geq0}\frac{q^{s_1^2+\dots+s_{r-1}^2+s_{i}+\cdots+ s_{r-2}+2s_{r-1}}}{(q)_{s_1-s_2}\cdots(q)_{s_{r-2}-s_{r-1}}\bigl(q^2;q^2\bigr)_{s_{r-1}}}\nonumber\\
\qquad=\frac{1}{(q)_\infty}\left(\bigl(q^{2r},q^{2r-i-1},q^{i+1};q^{2r}\bigr)_\infty+q\bigl(q^{2r},q^{2r-i+1},q^{i-1};q^{2r}\bigr)_\infty\right),\label{fij0}
\end{gather}
where $r \geq 2$ and $1\leq i\leq r$, and
\begin{gather}
\sum_{s_1\geq\dots\geq s_{r-1}\geq0}\frac{q^{s_1^2+\dots+s_{r-1}^2-s_1-\cdots- s_i+s_{r-1}}}{(q)_{s_1-s_2}\cdots(q)_{s_{r-2}-s_{r-1}}\bigl(q^2;q^2\bigr)_{s_{r-1}}}\nonumber\\
\qquad=\sum_{k=0}^{i}\frac{\bigl(q^{2r},q^{r-i+2k-1},q^{r+i-2k+1};q^{2r}\bigr)_\infty}{(q)_\infty},\label{fij}
\end{gather}
where $r \geq 2$ and $0\leq i\leq r-1$.

Again, we are able to embed \eqref{fij0} and \eqref{fij} into a general $m$-version, namely Theorem~\ref{thm:mfij}. To do this, instead of~\eqref{bilatD4}, we use the following bilateral Bailey pair generalising (D1) in~\cite{BIS}, and which is considered in~\cite{J}
\begin{equation}\label{bilatD1}
\alpha'_n=\alpha_n\bigl(a^2,q^2\bigr)\qquad\mbox{and}\qquad\beta'_n=\sum_{j\leq n}\frac{(-aq)_{2j}}{\bigl(q^2;q^2\bigr)_{n-j}}q^{n-j}\beta_j\bigl(a^2,q^2\bigr).
\end{equation}

\begin{proof}[Proof of Theorem \ref{thm:mfij}]
We start from the bilateral Bailey pair~\eqref{mbubp} with $a=q^m$, to which we apply~\eqref{bilatD1}. This results in the bilateral Bailey pair
\begin{equation*}
\alpha_n=(-1)^nq^{n^2-n}\qquad \mbox{and}\qquad\beta_n=\bigl(q^2;q^2\bigr)_m\sum_{j\leq n}(-1)^jq^{j^2+n-2j}\frac{(-q^{1+m})_{2j}}{\bigl(q^2;q^2\bigr)_{n-j}}\left[{m+j\atop m+2j}\right]_{q^2}.
\end{equation*}
Then apply Corollary~\ref{coro:bilatbaileylattice} with $a=q^m$ and $r$ replaced by $r-1$ to the above bilateral Bailey pair and divide both sides by $(q)_m$, the left-hand side is the desired one ($\beta_n$ above is $\beta_{s_{r-1}}$ while~${j=s_r}$). The right-hand side is equal to
\begin{equation*}
\frac{1}{(q)_\infty}\sum_{j\in\mathbb{Z}}(-1)^jq^{rj^2-(i+1)j+m(r-1)j}\frac{1-q^{(m+2j)(i+1)}}{1-q^{m+2j}},
\end{equation*}
which by expanding the denominator in a geometric series yields
\begin{gather*}
\frac{1}{(q)_\infty}\sum_{k=0}^iq^{mk}\sum_{j\in\mathbb{Z}}(-1)^{j}q^{rj^2-(i+1)j+m(r-1)j+2kj}\\
\qquad=\frac{1}{(q)_\infty}\sum_{k=0}^iq^{mk}\sum_{j\in\mathbb{Z}}(-1)^{j}q^{2r\binom{j}{2}}q^{j((m+1)(r-1)-i+2k)}.
\end{gather*}
This gives the result by using the Jacobi triple product identity~\eqref{jtp}.
\end{proof}

The case $i=0$ is \cite[Theorem~3.2]{J}. Taking $m=0$ in~\eqref{mfij} forces the index $s_r$ to be $0$, therefore we get~\eqref{fij}. Taking $m=1$ in~\eqref{mfij} also yields $s_r$ to be $0$, therefore the left-hand side is the one of~\eqref{fij0} (in which $i$ is replaced by $i+1$). Regarding the right-hand sides, the one of~\eqref{fij0} is given by the two first terms $k=0$ and $k=1$ in~\eqref{mfij} (with $i$ replaced by $i-1$), while the sum from $2$ to $i$ actually cancels, even though it is not immediate at first sight.

\subsection[m-versions of the Bressoud and Bressoud--Gollnitz--Gordon identities]{$\boldsymbol{m}$-versions of the Bressoud and Bressoud--G\"ollnitz--Gordon identities}

In~\cite{Br80}, in addition to~\eqref{AG} and~\eqref{Br3.3}--\eqref{Br3.5}, Bressoud proved four identities of the same kind, denoted (3.6)--(3.9) in his paper, among which (3.6)--(3.8) generalise the G\"ollnitz--Gordon identities of Theorem~\ref{th:GG}. In this section, we will give $m$-versions for all of these. More precisely, as our $m$-versions yield nice simplifications in the cases $m=0$ and $m=1$, all formulas come in pairs, as in the previous subsections. \cite[formulas~(3.6) and (3.7)]{Br80} will surprisingly arise in pairs with~\eqref{B} and~\eqref{Br3.5} respectively, while each of \cite[formulas~(3.8) and~(3.9)]{Br80} will be associated with formulas which seem to be new.

\subsubsection[m-version of Ref. 13, equation (3.6)]{$\boldsymbol{m}$-version of \cite[equation (3.6)]{Br80}}
First recall (3.6) in~\cite{Br80}
\begin{gather}
\sum_{s_1\geq\dots\geq s_{r-1}\geq0}\frac{q^{2(s_1^2+\dots+s_{r-1}^2-s_{1}-\dots-s_{i})}\bigl(-q^{1+2s_{r-1}};q^2\bigr)_\infty}{\bigl(q^2;q^2\bigr)_{s_1-s_2}\cdots
\bigl(q^2;q^2\bigr)_{s_{r-2}-s_{r-1}}\bigl(q^2;q^2\bigr)_{s_{r-1}}}\nonumber\\
\qquad=\frac{\bigl(-q;q^2\bigr)_\infty}{\bigl(q^2;q^2\bigr)_\infty}\sum_{k=0}^{i}\bigl(q^{4r},q^{2r-2i+2k-1},q^{2r+2i-2k+1};q^{4r}\bigr)_\infty,\label{B3.6}
\end{gather}
where $r\geq 2$ and $0\leq i\leq r-1$ are fixed integers. Note that the parameters in Bressoud's work are renamed $(k,r,i)\to(r,i+1,k)$ to match our notation. The appropriate $m$-version of this formula is given by~\eqref{mbr-3.6} that we prove below. Surprisingly, it also gives a $m$-version of~\eqref{B}.

\begin{proof}[Proof of Theorem~\ref{thm:mbr-3.6}]
We start from the bilateral Bailey pair~\eqref{mbubp} with $a=q^m$, to which we apply Corollary~\ref{coro:bilatbaileylattice2} with $a=q^m$, $b\to\infty$, $c=-q^{(m+1)/2}$ and divide both sides by $(q)_m$. The left-hand side is the desired one. The right-hand side is equal to
\begin{equation*}
\frac{1}{(q)_\infty}\sum_{j\in\mathbb{Z}}(-1)^jq^{rj^2-ij+mrj-(m+1)j/2}\frac{1-q^{(m+2j)(i+1)}}{1-q^{m+2j}},
\end{equation*}
which by expanding the denominator in a geometric series yields
\begin{gather*}
\frac{1}{(q)_\infty}\sum_{k=0}^iq^{mk}\sum_{j\in\mathbb{Z}}(-1)^{j}q^{rj^2-ij+mrj+2kj-(m+1)j/2}\\
\qquad=\frac{1}{(q)_\infty}\sum_{k=0}^iq^{mk}\sum_{j\in\mathbb{Z}}(-1)^{j}q^{2r\binom{j}{2}}q^{j((m+1)r-i+2k-(m+1)/2)}.
\end{gather*}
This gives the result by using the Jacobi triple product identity~\eqref{jtp}.
\end{proof}

Taking $m=0$ in~\eqref{mbr-3.6} forces the index $s_r$ to be $0$. Next, shifting $q\to q^2$ and multiplying both sides by
\begin{equation*}
\bigl(-q;q^2\bigr)_\infty=\bigl(-q;q^2\bigr)_{s_{r-1}}\bigl(-q^{1+2s_{r-1}};q^2\bigr)_\infty,
\end{equation*}
the left-hand side coincides with the one of~\eqref{B3.6}. The right-hand side is also the one of~\eqref{B3.6}: it is obvious for the even indices $2k$ on the right-hand side of~\eqref{B3.6} (for $0\leq 2k\leq i$), while the odd indices $2k+1$ correspond to indices $i-k$ on the right-hand side of~\eqref{mbr-3.6}. Taking $m=1$ in~\eqref{mbr-3.6} also yields $s_r$ to be $0$, therefore the left-hand side is the one of~\eqref{B} (in which $i$ is replaced by $i+1$). Regarding the right-hand sides, the one of~\eqref{B} (with $i$ replaced by $i+1$) is given by the first term $k=0$ in~\eqref{mbr-3.6}, while the sum from $1$ to $i$ actually cancels, even though it is not immediate at first sight.

\subsubsection[m-version of Ref.~13, equation (3.7)]{$\boldsymbol{m}$-version of \cite[equation (3.7)]{Br80}}
First recall (3.7) in~\cite{Br80}
\begin{gather}
\sum_{s_1\geq\dots\geq s_{r-1}\geq0}\frac{q^{2(s_1^2+\dots+s_{r-1}^2+s_{i+1}+\dots+s_{r-1})}\bigl(-q^{3+2s_{r-1}};q^2\bigr)_\infty}{\bigl(q^2;q^2\bigr)_{s_1-s_2}
\cdots\bigl(q^2;q^2\bigr)_{s_{r-2}-s_{r-1}}\bigl(q^2;q^2\bigr)_{s_{r-1}}}\nonumber\\
\qquad=\frac{\bigl(-q;q^2\bigr)_\infty}{\bigl(q^2;q^2\bigr)_\infty}\sum_{k=0}^{i}(-q)^k\bigl(q^{4r},q^{2i+1-2k},q^{4r-2i-1+2k};q^{4r}\bigr)_\infty,\label{B3.7}
\end{gather}
where $r\geq 2$ and $0\leq i\leq r-1$ are fixed integers. Note that the parameters in Bressoud's work are again renamed $(k,r,i)\to(r,i+1,k)$ to match our notation. Our $m$-version also extends~\eqref{Br3.5} and reads as follows.

\begin{Theorem}[$m$-version of the Bressoud identities {\cite[equation~(3.7)]{Br80}}]\label{thm:mbr-3.7}
Let $m\geq 0$, $r \geq 2$, and $0 \leq i \leq r-1$ be three integers. We have
\begin{gather}
\sum_{s_1\geq\dots\geq s_{r}\geq-\left\lfloor m/2\right\rfloor}\frac{q^{s_1^2+\dots+s_{r}^2+m(s_1+\dots+s_{r-1}+s_r/2)-s_1-\dots-s_{i}}\bigl(-q^{m/2}\bigr)_{s_r}}{(q)_{s_1-s_2} \cdots(q)_{s_{r-1}-s_{r}}\bigl(-q^{1+m/2}\bigr)_{s_{r-1}}}(-1)^{s_r}\left[{m+s_r\atop m+2s_r}\right]\nonumber\\
\qquad=b_m,\label{mbr-3.7}
\end{gather}
where
\begin{equation*}
b_{2m}=\frac{1+q^m}{2(q)_\infty}\sum_{k=0}^{i}\sum_{\ell=0}^{2m}(-1)^\ell q^{2mk+m\ell}\bigl(q^{2r},q^{2mr+r-i-m+2k+\ell},q^{r+i-2mr+m-2k-\ell};q^{2r}\bigr)_\infty,
\end{equation*}
and
\begin{align*}
b_{2m+1}={}&(-1)^mq^{(1-r)m^2+(1+2i-2r)m/2 }\frac{1+q^{(2m+1)/2}}{2(q)_\infty}\sum_{k=0}^i\sum_{\ell=0}^mq^{k+\ell}\\
&\times\bigl(\bigl(q^{2r},q^{2r-i-m+2k+2\ell-1/2},q^{i+m-2k-2\ell+1/2};q^{2r}\bigr)_\infty\\
&- q^{1/2}\bigl(q^{2r},q^{2r-i-m+2k+2\ell+1/2},q^{i+m-2k-2\ell-1/2};q^{2r}\bigr)_\infty\bigr).
\end{align*}
\end{Theorem}

\begin{proof}
We start from the bilateral Bailey pair~\eqref{mbubp} with $a=q^m$, to which we apply Corollary~\ref{coro:bilatbaileylattice2} with $a=q^m$, $b\to\infty$, $c=-q^{m/2}$ and divide both sides by $(q)_m$. The left-hand side is the desired one. The right-hand side is equal to
\begin{equation}\label{bm1}
b_m=\frac{1+q^{m/2}}{(q)_\infty}\sum_{j\in\mathbb{Z}}(-1)^jq^{rj^2-ij+mrj-mj/2}\frac{1-q^{(m+2j)(i+1)}}{1-q^{m+2j}}\frac{1}{1+q^{(m+2j)/2}}.
\end{equation}
As in the proof of Theorem~\ref{thm:mb}, shifting the index $j$ to $-j-m$ above and adding the result with~\eqref{bm1} yields after rearranging
\begin{equation}\label{bm0}
b_{m}=\frac{1+q^{m/2}}{2(q)_\infty}\sum_{j\in\mathbb{Z}}(-1)^jq^{rj^2-ij+mrj-mj/2}\frac{1-q^{(m+2j)(i+1)}}{1-q^{m+2j}}\frac{1+(-1)^mq^{(m+2j)(m+1)/2}}{1+q^{(m+2j)/2}}.
\end{equation}
This gives
\begin{equation*}
b_{2m}=\frac{1+q^m}{2(q)_\infty}\sum_{j\in\mathbb{Z}}(-1)^jq^{rj^2-ij+2mrj-mj}\frac{1-q^{(2m+2j)(i+1)}}{1-q^{2m+2j}}\frac{1+q^{(m+j)(2m+1)}}{1+q^{m+j}},
\end{equation*}
in which we can expand both denominators in geometric series and obtain the desired result by using~\eqref{jtp}. Equation~\eqref{bm0} also yields
\begin{align*}
b_{2m+1}={}&\frac{1+q^{(2m+1)/2}}{2(q)_\infty}\sum_{j\in\mathbb{Z}}(-1)^jq^{rj^2-ij+(2m+1)rj-(2m+1)j/2}\\
&\times\frac{1-q^{(2m+2j+1)(i+1)}}{1-q^{2m+2j+1}}\frac{1-q^{(2m+2j+1)(m+1)}}{1+q^{(2m+2j+1)/2}}.
\end{align*}
Using $\bigl(1+q^{(2m+2j+1)/2}\bigr)\bigl(1-q^{(2m+2j+1)/2}\bigr)=1-q^{2m+2j+1}$ and expanding both denominators in geometric series yields
\begin{align*}
b_{2m+1}={}&\frac{1+q^{(2m+1)/2}}{2(q)_\infty}\sum_{j\in\mathbb{Z}}(-1)^jq^{rj^2-ij+(2m+1)rj-(2m+1)j/2}\\
&\times\bigl(1-q^{(2m+2j+1)/2}\bigr)\sum_{k=0}^iq^{(2m+2j+1)k}\sum_{\ell=0}^mq^{(2m+2j+1)\ell},
\end{align*}
which, by using~\eqref{jtp}, yields
\begin{align*}
b_{2m+1}={}&\frac{1+q^{(2m+1)/2}}{2(q)_\infty}\sum_{k=0}^i\sum_{\ell=0}^mq^{(2m+1)(k+\ell)}\\
&\times\bigl(\bigl(q^{2r},q^{(2m+1)r-i+r-m+2k+2\ell-1/2},q^{r+i-(2m+1)r+m-2k-2\ell+1/2};q^{2r}\bigr)_\infty \\
 &- q^{(2m+1)/2}\bigl(q^{2r},q^{(2m+1)r-i+r-m+2k+2\ell+1/2},q^{r+i-(2m+1)r+m-2k-2\ell-1/2};q^{2r}\bigr)_\infty\bigr).
\end{align*}
The result follows after using manipulations similar to the ones for $a_{2m+1}$ in the proof of Theorem~\ref{thm:mb}.
\end{proof}

 Taking $m=0$ in~\eqref{mbr-3.7} forces the index $s_r$ to be $0$, therefore we obtain~\eqref{Br3.5}. Taking $m=1$ in~\eqref{mbr-3.7} also yields $s_r$ to be $0$. Next, shifting $q\to q^2$ and multiplying both sides by
\begin{equation*}
\bigl(-q^3;q^2\bigr)_\infty=\bigl(-q^3;q^2\bigr)_{s_{r-1}}\bigl(-q^{3+2s_{r-1}};q^2\bigr)_\infty,
\end{equation*}
the left-hand side coincides with the one of~\eqref{B3.7}. The right-hand side becomes
\begin{gather*}
\frac{\bigl(-q;q^2\bigr)_\infty}{2\bigl(q^2;q^2\bigr)_\infty}\sum_{k=0}^iq^{2k}\bigl(\bigl(q^{4r},q^{4r-2i+4k-1},q^{2i-4k+1};q^{4r}\bigr)_\infty
\\
\qquad{}-q\bigl(q^{4r},q^{4r-2i+4k+1},q^{2i-4k-1};q^{4r}\bigr)_\infty\bigr).
\end{gather*}
This sum is indeed twice the one on the right-hand side of~\eqref{B3.7}: to see this, keep the terms~$k$ above for $0\leq 2k\leq i$ and $0\leq 2k+1\leq i$, and replace $k$ by $i-k$ for the terms $k$ satisfying~${i+1\leq 2k\leq 2i}$ and $i+1\leq 2k+1\leq 2i+1$.

\subsubsection[m-version of Ref. 13, equation (3.8)]{$\boldsymbol{m}$-version of~\cite[equation (3.8)]{Br80}}
First recall (3.8) in~\cite{Br80}
\begin{gather}
\sum_{s_1\geq\dots\geq s_{r-1}\geq0}\frac{q^{2(s_1^2+\dots+s_{r-1}^2+s_{i+1}+\dots+s_{r-1})}\bigl(-q^{1-2s_{1}};q^2\bigr)_{s_1}}{\bigl(q^2;q^2\bigr)_{s_1-s_2}
\cdots\bigl(q^2;q^2\bigr)_{s_{r-2}-s_{r-1}}\bigl(q^2;q^2\bigr)_{s_{r-1}}}\nonumber\\
\qquad=\frac{\bigl(-q;q^2\bigr)_\infty}{\bigl(q^2;q^2\bigr)_\infty}\bigl(q^{4r},q^{2i+1},q^{4r-2i-1};q^{4r}\bigr)_\infty,\label{B3.8}
\end{gather}
where $r\geq 2$ and $0\leq i\leq r-1$ are fixed integers. Note that the parameters in Bressoud's work are again changed by $(k,r,i)\to(r,i+1,k)$ to match our notation. Our $m$-version reads as follows.

\begin{Theorem}[$m$-version of the Bressoud identities {\cite[equation (3.8)]{Br80}}]\label{thm:mbr-3.8}
Let $m\geq 0$, $r \geq 2$, and $0 \leq i \leq r-1$ be three integers. We have
\begin{gather}
\sum_{s_1\geq\dots\geq s_{r}\geq-\left\lfloor m/2\right\rfloor}\frac{q^{s_1^2/2+s_2^2+\dots+s_{r}^2+m(s_1/2+s_2+\dots+s_{r-1})+s_1/2-(s_1+\dots+s_{i})}\bigl(-q^{m/2}\bigr)_{s_1}}{(q)_{s_1-s_2}\cdots(q)_{s_{r-1}-s_{r}}}\nonumber\\
\qquad{}\times(-1)^{s_r}q^{\binom{s_r}{2}}\left[{m+s_r\atop m+2s_r}\right]=c_m,\label{mbr-3.8}
\end{gather}
where
\begin{equation*}c_{2m}=\frac{(-q^m)_\infty}{2(q)_\infty}\sum_{k=0}^{2i}\sum_{\ell=0}^{2m}(-1)^\ell q^{mk+m\ell}\bigl(q^{2r},q^{2mr+r-i-m+k+\ell},q^{r+i-2mr+m-k-\ell};q^{2r}\bigr)_\infty,
\end{equation*}
and
\begin{align*}
c_{2m+1}={}&(-1)^mq^{(1-r)m^2+(1+2i-2r)m/2}\frac{\bigl(-q^{(2m+1)/2}\bigr)_\infty}{(q)_\infty}\\
&\times\sum_{\ell=0}^{m}q^{\ell}\bigl(q^{2r},q^{2r-i-m+2\ell-1/2},q^{i+m-2\ell+1/2};q^{2r}\bigr)_\infty.
\end{align*}
\end{Theorem}

\begin{proof}
We start from the bilateral Bailey pair~\eqref{mbubp} with $a=q^m$, to which we apply Corollary~\ref{coro:bilatbaileylattice2} with $a=q^m$, $b=-q^{m/2}$, $c\to\infty$ and divide both sides by $(q)_m$. The left-hand side is the desired one. The right-hand side is equal to
\begin{equation}\label{cm1}
c_m=\frac{\bigl(-q^{m/2}\bigr)_\infty}{(q)_\infty}\sum_{j\in\mathbb{Z}}(-1)^jq^{rj^2-ij+mrj-mj/2}\frac{1-q^{(m+2j)(2i+1)/2}}{1-q^{m+2j}}.
\end{equation}
As in the proof of Theorem~\ref{thm:mb}, shifting the index $j$ to $-j-m$ above and adding the result with~\eqref{cm1} yields after rearranging
\begin{align}
c_{m}={}&\frac{\bigl(-q^{m/2}\bigr)_\infty}{2(q)_\infty}\sum_{j\in\mathbb{Z}}(-1)^jq^{rj^2-ij+mrj-mj/2}\frac{1-q^{(m+2j)(2i+1)/2}}{1-q^{m+2j}}\nonumber\\
&\times\bigl(1+(-1)^mq^{(m+2j)(m+1)/2}\bigr).\label{cm0}
\end{align}
This gives
\begin{equation*}
c_{2m}=\frac{(-q^{m})_\infty}{2(q)_\infty}\sum_{j\in\mathbb{Z}}(-1)^jq^{rj^2-ij+2mrj-mj}\frac{1-q^{(m+j)(2i+1)}}{1-q^{m+j}}\frac{1+q^{(m+j)(2m+1)}}{1+q^{m+j}},
\end{equation*}
in which we can expand both denominators in geometric series and obtain the desired result by using~\eqref{jtp}. Equation~\eqref{cm0} also yields
\begin{align*}
c_{2m+1}={}&\frac{\bigl(-q^{(2m+1)/2}\bigr)_\infty}{2(q)_\infty}\sum_{j\in\mathbb{Z}}(-1)^jq^{rj^2-ij+(2m+1)rj-(2m+1)j/2}\\
&\times\bigl(1-q^{(2m+2j+1)(2i+1)/2}\bigr)\frac{1-q^{(2m+2j+1)(m+1)}}{1-q^{2m+2j+1}},
\end{align*}
and the result follows after expanding the denominator in a geometric series and using manipulations similar to the ones for $a_{2m+1}$ in the proof of Theorem~\ref{thm:mb}.
\end{proof}

 Taking $m=0$ in~\eqref{mbr-3.8} forces the index $s_r$ to be $0$, therefore we obtain the following identity, which seems to be new.
\begin{Corollary}\label{coro:new1}
Let $r \geq 2$ and $0 \leq i \leq r-1$ be two integers. We have
\begin{gather*}%\label{3.8new}
\sum_{s_1\geq\dots\geq s_{r-1}\geq0}\frac{q^{s_1^2/2+s_2^2+\dots+s_{r-1}^2+s_1/2-(s_1+\dots+s_{i})}(-1)_{s_1}}{(q)_{s_1-s_2}\cdots(q)_{s_{r-2}-s_{r-1}}(q)_{s_{r-1}}} \\
\qquad=\frac{(-q)_\infty}{(q)_\infty}\sum_{k=0}^{2i}\bigl(q^{2r},q^{r-i+k},q^{r+i-k};q^{2r}\bigr)_\infty.
\end{gather*}
\end{Corollary}

Taking $m=1$ in~\eqref{mbr-3.8} also yields $s_r$ to be $0$. Replacing $q$ by $q^2$ in the resulting formula therefore yields~\eqref{B3.8} by using
$\bigl(-q;q^2\bigr)_{s_1}=q^{s_1^2}\bigl(-q^{1-2s_{1}};q^2\bigr)_{s_1}$.

\subsubsection[m-version of Ref. 13, equation (3.9)]{$\boldsymbol{m}$-version of \cite[equation (3.9)]{Br80}}
First recall (3.9) in~\cite{Br80}
\begin{gather}
\sum_{s_1\geq\dots\geq s_{r-1}\geq0}\frac{q^{2(s_1^2+\dots+s_{r-1}^2+s_{i+1}+\dots+s_{r-1})}\bigl(-q^{1-2s_{1}};q^2\bigr)_{s_1}}{\bigl(q^2;q^2\bigr)_{s_1-s_2}\cdots
\bigl(q^2;q^2\bigr)_{s_{r-2}-s_{r-1}}\bigl(q^4;q^4\bigr)_{s_{r-1}}}\nonumber\\
\qquad=\frac{\bigl(-q;q^2\bigr)_\infty}{\bigl(q^2;q^2\bigr)_\infty}(q^{4r-2},q^{2i+1},q^{4r-2i-3};q^{4r-2})_\infty,\label{B3.9}
\end{gather}
where $r\geq 2$ and $0\leq i\leq r-1$ are fixed integers. Note that the parameters in Bressoud's work are again renamed $(k,r,i)\to(r,i+1,k)$ to match our notation. Our $m$-version reads as follows.

\begin{Theorem}[$m$-version of the Bressoud identities {\cite[equation (3.9)]{Br80}}]\label{thm:mbr-3.9}
Let $m\geq 0$, $r \geq 2$, and $0 \leq i \leq r-1$ be three integers. We have
\begin{gather}
\sum_{s_1\geq\dots\geq s_{r}\geq-\left\lfloor m/2\right\rfloor}\frac{q^{s_1^2/2+s_2^2+\dots+s_{r}^2+m(s_1/2+s_2+\dots+s_{r-1})+s_1/2-(s_1+\dots+s_{i})}\bigl(-q^{m/2}\bigr)_{s_1} \bigl(-q^{(m+1)/2}\bigr)_{s_{r}}}{(q)_{s_1-s_2}
\cdots(q)_{s_{r-1}-s_{r}}\bigl(-q^{(m+1)/2}\bigr)_{s_{r-1}}}\nonumber\\
\qquad{}\times(-1)^{s_r}q^{-(m+1)s_r/2}\left[{m+s_r\atop m+2s_r}\right]
=d_m,\label{mbr-3.9}
\end{gather}
where
\begin{align*}
d_{2m}={}&\frac{(-q^m)_\infty}{2(q)_\infty}\sum_{k=0}^{2i}\sum_{\ell=0}^{2m}(-1)^\ell q^{mk+m\ell}\\
&\times(q^{2r-1},q^{2mr+r-i-2m+k+\ell-1/2},q^{r+i+2m-2mr-k-\ell-1/2};q^{2r-1})_\infty,
\end{align*}
and
\begin{align*}
d_{2m+1}={}&(-1)^mq^{(3-2r)m^2/2+(1+i-r)m}\frac{\bigl(-q^{(2m+1)/2}\bigr)_\infty}{(q)_\infty}\\
&\times\sum_{\ell=0}^{m}q^{\ell}\bigl(q^{2r-1},q^{2r-i-m+2\ell-3/2},q^{i+m-2\ell+1/2};q^{2r-1}\bigr)_\infty.
\end{align*}
\end{Theorem}

\begin{proof}
We start from the bilateral Bailey pair~\eqref{mbubp} with $a=q^m$, to which we apply Corollary~\ref{coro:bilatbaileylattice2} with $a=q^m$, $b=-q^{m/2}$, $c=q^{(m+1)/2}$ and divide both sides by $(q)_m$. The left-hand side is the desired one. The right-hand side is equal to
\begin{equation}\label{dm1}
d_m=\frac{\bigl(-q^{m/2}\bigr)_\infty}{(q)_\infty}\sum_{j\in\mathbb{Z}}(-1)^jq^{(r-1)j^2-ij+m(r-1)j+j^2/2}\frac{1-q^{(m+2j)(2i+1)/2}}{1-q^{m+2j}}.
\end{equation}
As in the proof of Theorem~\ref{thm:mb}, shifting the index $j$ to $-j-m$ above and adding the result with~\eqref{dm1} yields after rearranging
\begin{align}
d_{m}={}&\frac{\bigl(-q^{m/2}\bigr)_\infty}{2(q)_\infty}\sum_{j\in\mathbb{Z}}(-1)^jq^{(r-1)j^2-ij+m(r-1)j+j^2/2}\frac{1-q^{(m+2j)(2i+1)/2}}{1-q^{m+2j}}\nonumber\\
&\times\bigl(1+(-1)^mq^{(m+2j)(m+1)/2}\bigr).\label{dm0}
\end{align}
This gives
\begin{equation*}
d_{2m}=\frac{(-q^{m})_\infty}{2(q)_\infty}\sum_{j\in\mathbb{Z}}(-1)^jq^{(r-1)j^2-ij+2m(r-1)j+j^2/2}\frac{1-q^{(m+j)(2i+1)}}{1-q^{m+j}}\frac{1+q^{(m+j)(2m+1)}}{1+q^{m+j}},
\end{equation*}
in which we can expand both denominators in geometric series and obtain the desired result by using~\eqref{jtp}. Equation~\eqref{cm0} also yields
\begin{align*}
d_{2m+1}={}&\frac{\bigl(-q^{(2m+1)/2}\bigr)_\infty}{2(q)_\infty}\sum_{j\in\mathbb{Z}}(-1)^jq^{(r-1)j^2-ij+(2m+1)(r-1)j+j^2/2}\\
&\times\bigl(1-q^{(2m+2j+1)(2i+1)/2}\bigr)\frac{1-q^{(2m+2j+1)(m+1)}}{1-q^{2m+2j+1}},
\end{align*}
and the result follows after expanding the denominator in a geometric series and using manipulations similar to the ones for $a_{2m+1}$ in the proof of Theorem~\ref{thm:mb}.
\end{proof}

 Taking $m=0$ in~\eqref{mbr-3.9} forces the index $s_r$ to be $0$, therefore we obtain the following identity, which seems to be new.
\begin{Corollary}\label{coro:new2}
Let $r \geq 2$ and $0 \leq i \leq r-1$ be two integers. We have
\begin{gather*}%\label{3.9new}
\sum_{s_1\geq\dots\geq s_{r-1}\geq0}\frac{q^{s_1^2/2+s_2^2+\dots+s_{r-1}^2+s_1/2-(s_1+\dots+s_{i})}(-1)_{s_1}}{(q)_{s_1-s_2}\cdots(q)_{s_{r-2}-s_{r-1}}(q)_{s_{r-1}}\bigl(-q^{1/2}\bigr)_{s_{r-1}}} \\
\qquad=\frac{(-q)_\infty}{(q)_\infty}\sum_{k=0}^{2i}\bigl(q^{2r-1},q^{r-i+k-1/2},q^{r+i-k-1/2};q^{2r-1}\bigr)_\infty.
\end{gather*}
\end{Corollary}

Taking $m=1$ in~\eqref{mbr-3.9} also yields $s_r$ to be $0$. Replacing $q$ by $q^2$ in the resulting formula therefore yields~\eqref{B3.9}.

\section[Bilateral N-extensions]{Bilateral $\boldsymbol{N}$-extensions}\label{sec:newBailey}

\subsection{Results}
Using our new bilateral Bailey lattice given in Theorem \ref{thm:newbilatbaileylattice} and Lemma \ref{lem:baileylattices}, we were able to deduce Theorem \ref{thm:multibaileylattice}, a very general bilateral $N$-Bailey lattice with parameters $b_1, \dots , b_N$.

The proof of Theorem \ref{thm:multibaileylattice}, quite technical, is left for the next subsection. However, some of its particular cases, which correspond to the two key Lemmas \ref{lem:key1} and \ref{lem:key2}, are much more simple to state (and to prove), and imply two bilateral $N$-extensions of the Bailey lattice found by Warnaar in~\cite[Theorems~3.1 and~3.2]{W}. Hence we state them separately here.

Since $e_M(0,\dots,0)=\delta_{M,0}$,
\begin{equation*}
f_{N,j,n}(0,\ldots,0) =\sum_{u\in \mathbb{Z}} a^{j-u}q^{(0-u)(n-u)+(j-u)(n-N)} \begin{bmatrix}
0\\
u
\end{bmatrix} \begin{bmatrix}
N\\
j-u
\end{bmatrix} = a^j q^{j(n-N)}\begin{bmatrix}
N\\
j
\end{bmatrix},
\end{equation*}
and Theorem~\ref{thm:multibaileylattice} reduces to the following.

\begin{Theorem}[first new $N$-Bailey lattice]\label{thm:Nbaileylattice1}
Let $(\alpha_n, \beta_n)$ be a bilateral Bailey pair relative to $a$. For all $N\geq 0$, define the pair \smash{$\bigl(\alpha^{(N)}_n, \beta^{(N)}_n\bigr)$} by
\[%\label{eq:multibaileylattice1}
\alpha^{(N)}_n=\bigl(1-aq^{2n-N}\bigr)\bigl(aq^{1-N}\bigr)_N\sum_{j\in \mathbb{Z}}(-1)^j \frac{a^j q^{(2n-N)j-j(j+1)/2}}{\bigl(aq^{2n-N-j}\bigr)_{N+1}}\left[{N\atop j}\right]\alpha_{n-j},
\]
and
\smash{$
\beta^{(N)}_n=\beta_n$}.
Then \smash{$\bigl(\alpha^{(N)}_n, \beta^{(N)}_n\bigr)$} is a bilateral Bailey pair relative to $aq^{-N}$.
\end{Theorem}

Applying first Theorem~\ref{thm:Nbaileylattice1} to a bilateral Bailey pair relative to $a$, and then Theorem~\ref{thm:bilatbaileylemma} with $a$ replaced by $aq^{-N}$ to the resulting Bailey pair, we immediately derive the following result, whose unilateral case is due to Warnaar~\cite[Theorem~3.1]{W}.

\begin{Theorem}[Warnaar, bilateral version]\label{thm:w1}
Let $(\alpha_n, \beta_n)$ be a bilateral Bailey pair relative to $a$, and $N\geq 0$ be a fixed integer. Then $(\alpha'_n, \beta'_n)$ is a bilateral Bailey pair relative to $aq^{-N}$, where
\begin{align*}
\alpha'_n={}&\frac{(\rho,\sigma)_n\bigl(aq^{1-N}/\rho\sigma\bigr)^n}{
\bigl(aq^{1-N}/\rho,aq^{1-N}/\sigma\bigr)_n}\bigl(1-aq^{2n-N}\bigr)\bigl(aq^{1-N}\bigr)_N\\
&\times\sum_{j=0}^N(-1)^j\frac{a^jq^{(2n-N)j-j(j+1)/2}}{\bigl(aq^{2n-N-j}\bigr)_{N+1}}\left[{N\atop j}\right]\alpha_{n-j},
\end{align*}
and
\begin{equation*}
\beta'_n=\sum_{j=0}^n\frac{(\rho,\sigma)_j\bigl(aq^{1-N}/\rho\sigma\bigr)_{n-j}\bigl(aq^{1-N}/\rho\sigma\bigr)^j}{(q)_{n-j}\bigl(aq^{1-N}/\rho,aq^{1-N}/\sigma\bigr)_n} \beta_j.
\end{equation*}
\end{Theorem}

Note that the bilateral Bailey lattice given in Theorem~\ref{thm:bilatbaileylattice} corresponds to the case $N=1$ in Theorem \ref{thm:w1}.

On the other hand, applying first Theorem~\ref{thm:bilatbaileylemma} to a bilateral Bailey pair relative to $a$, and then Theorem~\ref{thm:Nbaileylattice1} to the resulting bilateral Bailey pair, we derive the following second result, whose unilateral version is also due to Warnaar~\cite[Theorem~3.2]{W}.
\begin{Theorem}[Warnaar, bilateral version]\label{thm:w2}
Let $(\alpha_n, \beta_n)$ be a bilateral Bailey pair relative to $a$, and $N\geq 0$ be a fixed integer. Then $(\alpha'_n, \beta'_n)$ is a bilateral Bailey pair relative to $aq^{-N}$, where
\begin{align*}
\alpha'_n={}&\bigl(1-aq^{2n-N}\bigr)\bigl(aq^{1-N}\bigr)_N\sum_{j=0}^N(-1)^j\\
&\times\frac{a^jq^{(2n-N)j-j(j+1)/2}}{\bigl(aq^{2n-N-j}\bigr)_{N+1}}\left[{N\atop j}\right]\frac{(\rho,\sigma)_{n-j}(aq/\rho\sigma)^{n-j}}{
(aq/\rho,aq/\sigma)_{n-j}}\alpha_{n-j},
\end{align*}
and
\begin{equation*}
\beta'_n=\sum_{j=0}^n\frac{(\rho,\sigma)_j(aq/\rho\sigma)_{n-j}(aq/\rho\sigma)^j}{(q)_{n-j}(aq/\rho,aq/\sigma)_n} \beta_j.
\end{equation*}
\end{Theorem}

 Since $\lim_{b\to\infty} (1-b)^{-N} e_M(b,\dots,b)=\delta_{M,N}$,
\begin{align*}
\lim_{b \to\infty} \frac{f_{N,j,n}(b,\ldots,b)}{(1-b)^N} &=\sum_{u\in \mathbb{Z}} a^{j-u}q^{(N-u)(n-u)+(j-u)(n-N)} \begin{bmatrix}
N\\
u
\end{bmatrix} \times \begin{bmatrix}
0\\
j-u
\end{bmatrix} = q^{(N-j)(n-j)}\begin{bmatrix}
N\\
j
\end{bmatrix},
\end{align*}
and Theorem \ref{thm:multibaileylattice} reduces to the following.

\begin{Theorem}[second new $N$-Bailey lattice]\label{thm:Nbaileylattice2}
Let $(\alpha_n, \beta_n)$ be a bilateral Bailey pair relative to $a$. For all $N\geq 0$, define the pair \smash{$\bigl(\alpha^{(N)}_n, \beta^{(N)}_n\bigr)$} by
\[%\label{eq:multibaileylattice2}
\alpha^{(N)}_n=\bigl(1-aq^{2n-N}\bigr)\bigl(aq^{1-N}\bigr)_N\sum_{j\in \mathbb{Z}}(-1)^j\frac{q^{N(n-j)+j(j-1)/2}}{\bigl(aq^{2n-N-j}\bigr)_{N+1}}\left[{N\atop j}\right]\alpha_{n-j},
\]
and
\smash{$
\beta^{(N)}_n=q^{nN}\beta_n$}.
Then \smash{$\bigl(\alpha^{(N)}_n, \beta^{(N)}_n\bigr)$} is a bilateral Bailey pair relative to $aq^{-N}$.
\end{Theorem}

We give two new theorems similar to Warnaar's $N$-Bailey lattices, but coming from Theorem~\ref{thm:Nbaileylattice2} instead of Theorem \ref{thm:Nbaileylattice1}.

Applying first Theorem~\ref{thm:Nbaileylattice2} to a Bailey pair relative to $a$, and then Theorem~\ref{thm:bilatbaileylemma} with $a$ replaced by $aq^{-N}$ to the resulting bilateral Bailey pair gives the following result.

\begin{Theorem}\label{thm:analogw1}
Let $(\alpha_n, \beta_n)$ be a bilateral Bailey pair relative to $a$, and $N\geq 0$ be a fixed integer. Then $(\alpha'_n, \beta'_n)$ is a bilateral Bailey pair relative to $aq^{-N}$, where
\begin{equation*}
\alpha'_n=\frac{(\rho,\sigma)_n\bigl(aq^{1-N}/\rho\sigma\bigr)^n}{
\bigl(aq^{1-N}/\rho,aq^{1-N}/\sigma\bigr)_n}\bigl(1-aq^{2n-N}\bigr)\bigl(aq^{1-N}\bigr)_N\sum_{j=0}^N(-1)^j\frac{q^{N(n-j)+j(j-1)/2}}{\bigl(aq^{2n-N-j}\bigr)_{N+1}}\left[{N\atop j}\right]\alpha_{n-j},
\end{equation*}
and
\begin{equation*}
\beta'_n=\sum_{j=0}^n\frac{(\rho,\sigma)_j\bigl(aq^{1-N}/\rho\sigma\bigr)_{n-j}\bigl(aq^{1-N}/\rho\sigma\bigr)^j}{(q)_{n-j}\bigl(aq^{1-N}/\rho,aq^{1-N}/\sigma\bigr)_n} q^{jN}\beta_j.
\end{equation*}
\end{Theorem}

Note that Theorem~\ref{thm:newbilatbaileylattice} corresponds to the case $N=1$ of Theorem \ref{thm:analogw1}.

On the other hand, applying first Theorem~\ref{thm:bilatbaileylemma} to a bilateral Bailey pair relative to $a$, and then Theorem~\ref{thm:Nbaileylattice2} to the resulting bilateral Bailey pair, we derive the following result.

\begin{Theorem}\label{thm:analogw2}
Let $(\alpha_n, \beta_n)$ be a bilateral Bailey pair relative to $a$, and $N\geq 0$ be a fixed integer. Then $(\alpha'_n, \beta'_n)$ is a bilateral Bailey pair relative to $aq^{-N}$, where
\begin{equation*}
\alpha'_n=\bigl(1-aq^{2n-N}\bigr)\bigl(aq^{1-N}\bigr)_N\sum_{j=0}^N(-1)^j\frac{q^{N(n-j)+j(j-1)/2}}{\bigl(aq^{2n-N-j}\bigr)_{N+1}}\left[{N\atop j}\right]\frac{(\rho,\sigma)_{n-j}(aq/\rho\sigma)^{n-j}}{
(aq/\rho,aq/\sigma)_{n-j}}\alpha_{n-j},
\end{equation*}
and
\begin{equation*}
\beta'_n=q^{nN}\sum_{j=0}^n\frac{(\rho,\sigma)_j(aq/\rho\sigma)_{n-j}(aq/\rho\sigma)^j}{(q)_{n-j}(aq/\rho,aq/\sigma)_n} \beta_j.
\end{equation*}
\end{Theorem}

\subsection[Proof of Theorem~\ref{thm:multibaileylattice}]{Proof of Theorem~\ref{thm:multibaileylattice}}

Now we can turn to the proof of Theorem~\ref{thm:multibaileylattice}.
Recall the classical $q$-analogues of Pascal's triangle
\begin{equation}\label{eq:pascal1}
\left[{N+1\atop j}\right] = q^j\left[{N\atop j}\right]+\left[{N\atop j-1}\right],
\end{equation}
and
\begin{equation} \label{eq:pascal2}
\left[{N+1\atop j}\right]=\left[{N\atop j}\right]+ q^{N+1-j}\left[{N\atop j-1}\right],
\end{equation}
for all integers $N$, $j$ with $N\geq 0$.

\subsubsection[Recurrence relation for f\_\{N,j,n\}(b\_1,...,b\_N)]{Recurrence relation for $\boldsymbol{f_{N,j,n}(b_1,\ldots,b_N)}$}
We first prove the following recurrence relation, which plays a central role in our proof.
\begin{Proposition}[Recurrence relation]\label{prop:rec} For all $0\leq j\leq N+1$,
\begin{gather*}
\bigl(1-aq^{2n-1-N}\bigr)f_{N+1,j,n}(b_1,\ldots,b_{N+1})\\
\qquad =(1-b_{N+1}q^n)\bigl(1-aq^{2n-1-N-j}\bigr)f_{N,j,n}(b_1,\ldots,b_N)\\
\phantom{\qquad =}{}+\bigl(aq^{n-1-N}-b_{N+1}\bigr)\bigl(1-aq^{2n-j}\bigr)f_{N,j-1,n-1}(b_1,\ldots,b_N).
\end{gather*}
\end{Proposition}

We start by giving two technical lemmas which are extensions of \eqref{eq:pascal1} and \eqref{eq:pascal2}. Their proofs are straightforward verifications and are therefore omitted here.

\begin{Lemma}
\label{lem:tech1}
For all $j,u \in \Z$ and $0\leq M\leq N$,
\begin{gather*}
\begin{bmatrix}
M\\
j-u
\end{bmatrix} \begin{bmatrix}
N+1-M\\
u
\end{bmatrix}-q^{2j-2u-M+1}\begin{bmatrix}
M\\
j-u+1
\end{bmatrix} \begin{bmatrix}
N+1-M\\
u-1
\end{bmatrix}\\
\qquad=q^{u}\begin{bmatrix}
M\\
j-u
\end{bmatrix} \begin{bmatrix}
N-M\\
u
\end{bmatrix}
+q^{j-u-M}\begin{bmatrix}
N-M\\
u-1
\end{bmatrix}\left(\begin{bmatrix}
M\\
j-u
\end{bmatrix}-\begin{bmatrix}
M\\
j-u+1
\end{bmatrix}\right)\\
\phantom{\qquad=}{} -q^{2j-3u-2M+3+N}\begin{bmatrix}
M\\
j-u+1
\end{bmatrix} \begin{bmatrix}
N-M\\
u-2
\end{bmatrix} .
\end{gather*}
\end{Lemma}

Note that when $M=0$ and $u=j$, Lemma~\ref{lem:tech1} reduces to \eqref{eq:pascal1}.

\begin{Lemma}
\label{lem:tech2}
For all $j,u \in \Z$ and $0\leq M\leq N$,
\begin{gather*}
\begin{bmatrix}
M+1\\
j-u
\end{bmatrix} \begin{bmatrix}
N-M\\
u
\end{bmatrix}-q^{2j-2u-M}\begin{bmatrix}
M+1\\
j-u+1
\end{bmatrix} \begin{bmatrix}
N-M\\
u-1
\end{bmatrix}\\
\qquad
=q^j \begin{bmatrix}
M\\
j-u
\end{bmatrix} \begin{bmatrix}
N-M\\
u
\end{bmatrix}
- q^{2j-2u-M}\begin{bmatrix}
M\\
j-u+1
\end{bmatrix} \begin{bmatrix}
N-M\\
u-1
\end{bmatrix}+\begin{bmatrix}
M\\
j-u-1
\end{bmatrix} \begin{bmatrix}
N-M\\
u
\end{bmatrix}\\
\phantom{\qquad=}{}- q^{N+j+1-2u-M} \begin{bmatrix}
M\\
j-u
\end{bmatrix} \begin{bmatrix}
N-M\\
u-1
\end{bmatrix}.
\end{gather*}
\end{Lemma}

Note that when $M=1$ and $u=j$, Lemma~\ref{lem:tech2} reduces to a combination of \eqref{eq:pascal1} and \eqref{eq:pascal2}.

We can now prove our recurrence relation.
\begin{proof}[Proof of Proposition \ref{prop:rec}]
Using the homogeneity of $e_M$, recall from Theorem~\ref{thm:multibaileylattice} that
\begin{gather*}
f_{N,j,n}(b_1,\ldots,b_N)\\
\qquad= \sum_{M \in \Z} \sum_{u \in \Z} a^{u}q^{(M-j+u)(n-j+u)+u(n-N)} \begin{bmatrix}
M\\
j-u
\end{bmatrix} \begin{bmatrix}
N-M\\
u
\end{bmatrix} e_M(-b_1, \dots , -b_N).
\end{gather*}
Thus
\begin{gather}
\bigl(1-aq^{2n-1-N}\bigr)f_{N+1,j,n}(b_1,\ldots,b_{N+1})\nonumber \\
\qquad= \sum_{M \in \Z} \sum_{u\in \mathbb{Z}} a^{u}q^{(M-j+u)(n-j+u)+u(n-N-1)}\nonumber\\
\phantom{\qquad=}{}\times\bigg( \begin{bmatrix}
M\\
j-u
\end{bmatrix} \begin{bmatrix}
N+1-M\\
u
\end{bmatrix}
-q^{2j-2u-M+1}\begin{bmatrix}
M\\
j-u+1
\end{bmatrix} \begin{bmatrix}
N+1-M\\
u-1
\end{bmatrix} \bigg)\nonumber\\
\phantom{\qquad=}{}\times e_M(-b_1, \dots , -b_{N+1}).\label{eq:lastline}
\end{gather}
Writing
\begin{equation*}e_M(-b_1, \dots , -b_{N+1})=e_M(-b_1, \dots , -b_{N})-b_{N+1}e_{M-1}(-b_1, \dots , -b_{N}),
\end{equation*}
it is enough to show that coefficients of $b_{N+1}^0a^ue_M(-b_1, \dots , -b_{N})$ and $b_{N+1}^1a^ue_M(-b_1, \dots , -b_{N})$ coincide on both sides of Proposition~\ref{prop:rec}.

First, we derive from~\eqref{eq:lastline} that the coefficient of $b_{N+1}^0a^ue_M(-b_1, \dots , -b_{N})$ on the left-hand side of Proposition~\ref{prop:rec} is equal to \smash{$q^{(M-j+u)(n-j+u)+u(n-N-1)}$} times
\begin{equation}\label{eq:coeffb0l}
\begin{bmatrix}
M\\
j-u
\end{bmatrix} \begin{bmatrix}
N+1-M\\
u
\end{bmatrix}-q^{2j-2u-M+1}\begin{bmatrix}
M\\
j-u+1
\end{bmatrix} \begin{bmatrix}
N+1-M\\
u-1
\end{bmatrix}.
\end{equation}

Now the coefficient of $b_{N+1}^0a^ue_M(-b_1, \dots , -b_{N})$ on the right-hand side of Proposition~\ref{prop:rec} is equal to \smash{$q^{(M-j+u)(n-j+u)+u(n-N-1)}$} times
\begin{gather}
q^u \begin{bmatrix}
M\\
j-u
\end{bmatrix} \begin{bmatrix}
N-M\\
u
\end{bmatrix}-q^{j-u-M}\begin{bmatrix}
M\\
j-u+1
\end{bmatrix} \begin{bmatrix}
N-M\\
u-1
\end{bmatrix}+q^{j-u-M}\begin{bmatrix}
M\\
j-u
\end{bmatrix} \begin{bmatrix}
N-M\\
u-1
\end{bmatrix}\nonumber\\
\qquad{}-q^{N+2j-2M-3u+3} \begin{bmatrix}
M\\
j-u+1
\end{bmatrix} \begin{bmatrix}
N-M\\
u-2
\end{bmatrix}.\label{eq:coeffb0r}
\end{gather}
Note finally that~\eqref{eq:coeffb0l} (resp.~\eqref{eq:coeffb0r}) is the left-hand (resp.\ right-hand) side of Lemma~\ref{lem:tech1}, so we are done regarding this first coefficient.

Similarly, the coefficient of $b_{N+1}^1a^ue_M(-b_1, \dots , \!-b_{N})$ in~\eqref{eq:lastline} is \smash{$-q^{(M+1-j+u)(n-j+u)+u(n-N-1)}$} times
\begin{equation}\label{eq:coeffb1l}
\begin{bmatrix}
M+1\\
j-u
\end{bmatrix} \begin{bmatrix}
N-M\\
u
\end{bmatrix}-q^{2j-2u-M}\begin{bmatrix}
M+1\\
j-u+1
\end{bmatrix} \begin{bmatrix}
N-M\\
u-1
\end{bmatrix},
\end{equation}

and the coefficient of $b_{N+1}^1a^ue_M(-b_1, \dots , -b_{N})$ on the right-hand side of Proposition~\ref{prop:rec} is equal to \smash{$-q^{(M+1-j+u)(n-j+u)+u(n-N-1)}$} times
\begin{gather}
q^j \begin{bmatrix}
M\\
j-u
\end{bmatrix} \begin{bmatrix}
N-M\\
u
\end{bmatrix}-q^{2j-2u-M}\begin{bmatrix}
M\\
j-u+1
\end{bmatrix} \begin{bmatrix}
N-M\\
u-1
\end{bmatrix}+\begin{bmatrix}
M\\
j-u-1
\end{bmatrix} \begin{bmatrix}
N-M\\
u
\end{bmatrix}\nonumber\\
\qquad{}-q^{N+j+1-2u-M} \begin{bmatrix}
M\\
j-u
\end{bmatrix} \begin{bmatrix}
N-M\\
u-1
\end{bmatrix}.\label{eq:coeffb1r}
\end{gather}

As~\eqref{eq:coeffb1l} (resp.~\eqref{eq:coeffb1r}) is the left-hand (resp.\ right-hand) side of Lemma~\ref{lem:tech2}, this ends the proof of Proposition~\ref{prop:rec}.
\end{proof}

\subsubsection[Proof of Theorem \ref{thm:multibaileylattice}]{Proof of Theorem \ref{thm:multibaileylattice}}

Now we use Proposition \ref{prop:rec} and Lemma \ref{lem:baileylattices} to prove Theorem \ref{thm:multibaileylattice}.

We proceed by induction on $N$. For $N=0$, by \eqref{eq:multibaileylatticebeta}, \smash{$\beta^{(0)}_n=\beta_n$} and by \eqref{eq:multibaileylatticealpha},
\begin{equation*}
\alpha^{(0)}_n = \bigl(1-aq^{2n}\bigr)\frac{\alpha_{n}}{1-aq^{2n}}=\alpha_n.
\end{equation*}
Thus $\bigl(\alpha^{(0)}_n, \beta^{(0)}_n\bigr)$ is a bilateral Bailey pair relative to $a$.

Now, assume that for some integer $N\geq 0$, \smash{$\bigl(\alpha^{(N)}_n,\beta^{(N)}_n\bigr)$} is a bilateral Bailey pair relative to~$aq^{-N}$ and show that \smash{$\bigl(\alpha^{(N+1)}_n,\beta^{(N+1)}_n\bigr)$} is a bilateral Bailey pair relative to $aq^{-N-1}$. By Lemma~\ref{lem:baileylattices} with $b = b_{N+1}$, $(\alpha'_n,\beta'_n)$ is a bilateral Bailey pair relative to $aq^{-N-1}$, where
\begin{equation*}
\alpha'_n=\frac{1-aq^{-N}}{1-b_{N+1}}\left(\frac{(1-b_{N+1}q^n) \alpha_n^{(N)}}{1-aq^{2n-N}}-\frac{q^{n-1}\bigl(aq^{n-N-1}-b_{N+1}\bigr) \alpha_{n-1}^{(N)}}{1-aq^{2n-N-2}}\right),
\end{equation*}
and
\begin{equation*}\beta'_n=\frac{1-b_{N+1}q^n}{1-b_{N+1}}\beta_n^{(N)}.
\end{equation*}

Now let us show that $(\alpha'_n,\beta'_n) = \bigl(\alpha^{(N+1)}_n,\beta^{(N+1)}_n\bigr)$.
It is clear that $\beta'_n=\beta^{(N+1)}_n$, and for $\alpha'_n$, by~\eqref{eq:multibaileylatticealpha} we have
\begin{align*}
\alpha'_n
={}&\frac{1-aq^{-N}}{1-b_{N+1}}\left(\frac{(1-b_{N+1}q^n) \alpha_n^{(N)}}{1-aq^{2n-N}}-\frac{q^{n-1}\bigl(aq^{n-N-1}-b_{N+1}\bigr) \alpha_{n-1}^{(N)}}{1-aq^{2n-N-2}}\right)\\
={}& \frac{\bigl(aq^{-N}\bigr)_{N+1}}{(1-b_1)\cdots(1-b_{N+1})} \bigg(\sum_{j \in \Z}(-1)^j\frac{q^{jn-j(j+1)/2}(1-b_{N+1}q^n)f_{N,j,n}(b_1,\ldots,b_N)}{\bigl(aq^{2n-N-j}\bigr)_{N+1}} \alpha_{n-j} \\
& +\sum_{j \in \Z}(-1)^{j+1}\frac{q^{(j+1)(n-1)-j(j+1)/2}\bigl(aq^{n-1-N}-b_{N+1}\bigr)f_{N,j,n-1}(b_1,\ldots,b_N)}{\bigl(aq^{2n-2-N-j}\bigr)_{N+1}} \alpha_{n-1-j} \bigg)\\
={}& \frac{\bigl(aq^{-N}\bigr)_{N+1}}{(1-b_1)\cdots(1-b_{N+1})} \sum_{j \in \Z}(-1)^j\frac{q^{jn-j(j+1)/2}}{(aq^{2n-N-j-1})_{N+2}} \big( (1-b_{N+1}q^n)\bigl(1-aq^{2n-N-j-1}\bigr)
\\
& \times f_{N,j,n}(b_1,\ldots,b_N)+ \bigl(aq^{n-1-N}-b_{N+1}\bigr)\bigl(1-aq^{2n-j}\bigr)f_{N,j-1,n-1}(b_1,\ldots,b_N) \big) \alpha_{n-j}.
\end{align*}
Now by Proposition \ref{prop:rec}, this equals
\begin{align*}
\alpha'_n
={}& \frac{\bigl(aq^{-N}\bigr)_{N+1}}{(1-b_1)\cdots(1-b_{N+1})} \sum_{j \in \Z}(-1)^j\frac{q^{jn-j(j+1)/2}}{\bigl(aq^{2n-N-j-1}\bigr)_{N+2}} \bigl(1-aq^{2n-1-N}\bigr)\\
&\times f_{N+1,j,n}(b_1,\ldots,b_{N+1}) \alpha_{n-j}
= \alpha_n^{(N+1)}.
\end{align*}
Thus the pair \smash{$\bigl(\alpha^{(N+1)}_n, \beta^{(N+1)}_n\bigr)$} is indeed a Bailey pair relative to $aq^{-N-1}$.

\section{A new proof of Bressoud's identity}\label{sec:bressoud}

In this last section, we show that the unilateral version of Theorem~\ref{thm:newbilatbaileylattice} can be used to give a simple proof of Bressoud's identity \cite{Br80} which generalises the analytic version of the Rogers--Ramanujan identities (indeed for $k=2$, $r=1$ and $c_1,c_2,b_1\to\infty$, $a=1$ or $a=q$, \eqref{Bressoud2} reduces to the Rogers--Ramanujan identities).

\begin{Theorem}[Bressoud]\label{thm:Bressoud}
For integers $0<r<k$ and parameters $a$, $c_1$, $c_2$, $b_1,\dots, b_{2r-1}$, we have
\begin{gather}
\sum_{s_1\geq\dots\geq s_{k-1}\geq0}(-1)^{s_1}\frac{a^{s_1+\dots+s_{k-1}}q^{s_1^2/2+s_{r+1}^2+\dots+s_{k-1}^2-s_{1}/2+s_{r}}}{b_1^{s_1}(b_2b_{2r-1})^{s_2} \cdots(b_{r}b_{r+1})^{s_{r}}}\frac{(aq/c_1c_2)_{s_{k-1}}}{(q,aq/c_1,aq/c_2)_{s_{k-1}}}\nonumber\\
\qquad\times\frac{(b_1)_{s_1}(b_2,b_{2r-1})_{s_2}\cdots(b_{r},b_{r+1})_{s_{r}}}{(q)_{s_1-s_2}\cdots(q)_{s_{k-2}-s_{k-1}}}
\frac{(a/b_2b_{2r-1})_{s_1-s_2}\cdots(a/b_{r}b_{r+1})_{s_{r-1}-s_{r}}}{(a/b_2,a/b_{2r-1})_{s_1}\cdots(a/b_{r},a/b_{r+1})_{s_{r-1}}}\nonumber\\
\phantom{\qquad\times}{}=\frac{(a/b_1)_\infty}{(a)_\infty}\sum_{j\geq0}\frac{(b_1,\dots,b_{2r-1},c_1,c_2,a)_j(b_1\cdots b_{2r-1}c_1c_2)^{-j}a^{kj}q^{(k-r)j^2+j}}{(a/b_1,\dots,a/b_{2r-1},aq/c_1,aq/c_2,q)_j}\nonumber\\
\phantom{\qquad\times=}{}\times\left(1+\frac{a^{r}q^{j}}{b_1\cdots b_{2r-1}}\frac{\bigl(1-b_1q^j\bigr)\cdots\bigl(1-b_{2r-1} q^j\bigr)}{\bigl(1-aq^j/b_1\bigr)\cdots\bigl(1-aq^j/b_{2r-1}\bigr)}\right).\label{Bressoud2}
\end{gather}
\end{Theorem}
To see why the result above is actually equivalent to Bressoud's formula, see the following subsection.

The open problem of giving a combinatorial proof of Theorem \ref{thm:Bressoud} (when parameters $c_1,c_2\to\infty$ and the other parameters have specific forms), known as Bressoud's conjecture, was solved by Bressoud himself in some cases; then the next big step towards its resolution was made by Kim and Yee~\cite{KY}, and the full problem has recently been settled by Kim \cite{Ki}.

We give our proof of Theorem~\ref{thm:Bressoud} in this section, by showing that it is a consequence of the unilateral version of our new bilateral Bailey lattice in Theorem~\ref{thm:newbilatbaileylattice}. We also show that it does not seem to follow from the classical Bailey lattice of Theorem~\ref{thm:baileylattice}, which seems surprising at first sight.

%%%%%%%%%%%%%%%%%%
\subsection{Bressoud's result}
%%%%%%%%%%%%%%%%%%

In~\cite{Br80}, Bressoud defines $F_{\lambda,k,r}(c_1,c_2;b_1,\dots,b_\lambda;a;q)$ and $G_{\lambda,k,r}(c_1,c_2;b_1,\dots,b_\lambda;a;q)$, which are two functions whose integral parameters satisfy $k\geq r>\lambda/2\geq0$. This is equivalent to $2k-1\geq2r-1\geq\lambda\geq0$. Then Bressoud's main theorem in~\cite{Br80} states on the one hand that for all $k>r>\lambda/2\geq0$, we have
\begin{equation}\label{eq:bressoud}
F_{\lambda,k,r}(c_1,c_2;b_1,\dots,b_\lambda;a;q)=G_{\lambda,k,r}(c_1,c_2;b_1,\dots,b_\lambda;a;q),
\end{equation}
and on the other hand that for all $k\geq r>\lambda/2\geq0$, we have
\begin{equation}\label{eq:bressoudlim}
\lim_{c_1,c_2\to\infty}F_{\lambda,k,r}(c_1,c_2;b_1,\dots,b_\lambda;a;q)=\lim_{c_1,c_2\to\infty}G_{\lambda,k,r}(c_1,c_2;b_1,\dots,b_\lambda;a;q).
\end{equation}
We want to prove that these identities are both special cases of our new Bailey lattice. To do that, first note that it is enough to prove them when $\lambda$ takes its maximal value, that is $\lambda=2r-1$. Indeed, by the definitions of Bressoud's functions, we have
\begin{equation*}\lim_{b_\lambda\to\infty}F_{\lambda,k,r}(c_1,c_2;b_1,\dots,b_\lambda;a;q)=F_{\lambda-1,k,r}(c_1,c_2;b_1,\dots,b_{\lambda-1};a;q),
\end{equation*}
and
\begin{equation*}\lim_{b_\lambda\to\infty}G_{\lambda,k,r}(c_1,c_2;b_1,\dots,b_\lambda;a;q)=G_{\lambda-1,k,r}(c_1,c_2;b_1,\dots,b_{\lambda-1};a;q).
\end{equation*}

Now for $\lambda=2r-1$, one can define the first of these functions by
\begin{gather}
\frac{F_{2r-1,k,r}(c_1,c_2;b_1,\dots,b_{2r-1};a;q)}{(a/b_2,\dots,a/b_{2r-1})_\infty}\nonumber\\
\qquad=\frac{(a/b_1)_\infty}{(a)_\infty}\sum_{j\geq0}\frac{(b_1,\dots,b_{2r-1},c_1,c_2,a)_j(b_1\cdots b_{2r-1}c_1c_2)^{-j}a^{kj}q^{(k-r)j^2+j}}{(a/b_1,\dots,a/b_{2r-1},aq/c_1,aq/c_2,q)_j}\nonumber\\
\phantom{\qquad=}{}\times\left(1+\frac{a^{r}q^{j}}{b_1\cdots b_{2r-1}}\frac{\bigl(1-b_1q^j\bigr)\cdots\bigl(1-b_{2r-1} q^j\bigr)}{\bigl(1-aq^j/b_1\bigr)\cdots\bigl(1-aq^j/b_{2r-1}\bigr)}\right).\label{FBressoud}
\end{gather}

The second function of Bressoud can be defined for $\lambda=2r-1$ as
\begin{gather*}
G_{2r-1,k,r}(c_1,c_2;b_1,\dots,b_{2r-1};a;q)\\
\qquad=\sum_{s_1\geq\dots\geq s_{k-1}\geq0}\bigg(\frac{a^{s_1+\dots+s_{k-1}}q^{s_1^2+\dots+s_{k-1}^2-s_{1}-\dots-s_{r-1}}(aq/c_1c_2)_{s_{k-1}}}{(q,aq/c_1,aq/c_2)_{s_{k-1}}(q)_{s_1-s_2} \cdots(q)_{s_{k-2}-s_{k-1}}}\\
\phantom{\qquad=}{}\times\bigl(q^{1-s_1}/b_1\bigr)_{s_1}(a/b_2b_{2r-1})_{s_1-s_2}\cdots(a/b_rb_{r+1})_{s_{r-1}-s_r}\\
\phantom{\qquad=}{}\times\bigl(q^{1-s_2}/b_2,q^{1-s_2}/b_{2r-1}\bigr)_{s_2}\cdots\bigl(q^{1-s_r}/b_r,q^{1-s_r}/b_{r+1}\bigr)_{s_r}\\
\phantom{\qquad=}{}\times\bigl(aq^{s_1}/b_2,aq^{s_1}/b_{2r-1},\dots,aq^{s_{r-1}}/b_{r},aq^{s_{r-1}}/b_{r+1}\bigr)_\infty\bigg).
\end{gather*}

Using
\begin{gather*}
\bigl(q^{1-n}/b\bigr)_n=(-1)^nb^{-n}q^{-n(n-1)/2}(b)_n\qquad\mbox{and}\qquad(aq^n/b)_\infty=\frac{(a/b)_\infty}{(a/b)_n},
\end{gather*}
this gives
\begin{gather}
\frac{G_{2r-1,k,r}(c_1,c_2;b_1,\dots,b_{2r-1};a;q)}{(a/b_2,\dots,a/b_{2r-1})_\infty}\nonumber\\
\qquad=\sum_{s_1\geq\dots\geq s_{k-1}\geq0}\bigg((-1)^{s_1}\frac{a^{s_1+\dots+s_{k-1}}q^{s_1^2/2+s_{r+1}^2+\dots+s_{k-1}^2-s_{1}/2+s_{r}}}{b_1^{s_1}(b_2b_{2r-1})^{s_2}
\cdots(b_rb_{r+1})^{s_r}}\frac{(aq/c_1c_2)_{s_{k-1}}}{(q,aq/c_1,aq/c_2)_{s_{k-1}}}\nonumber\\
\phantom{\qquad=}{}\times \frac{(b_1)_{s_1}(b_2,b_{2r-1})_{s_2}\cdots(b_r,b_{r+1})_{s_r}}{(q)_{s_1-s_2}\cdots(q)_{s_{k-2}-s_{k-1}}}
\frac{(a/b_2b_{2r-1})_{s_1-s_2}\cdots(a/b_rb_{r+1})_{s_{r-1}-s_r}}{(a/b_2,a/b_{2r-1})_{s_1}\cdots(a/b_{r},a/b_{r+1})_{s_{r-1}}}\bigg).\label{GBressoud}
\end{gather}

Then identity~\eqref{eq:bressoud} of Bressoud translates for $\lambda=2r-1$ as
\begin{equation*}
F_{2r-1,k,r}(c_1,c_2;b_1,\dots,b_{2r-1};a;q)=G_{2r-1,k,r}(c_1,c_2;b_1,\dots,b_{2r-1};a;q),
\end{equation*}
with $0<r<k$, which from~\eqref{FBressoud} and~\eqref{GBressoud} is equivalent to \eqref{Bressoud2}.
Finally, Bressoud's second result~\eqref{eq:bressoudlim}
 asserts that~\eqref{Bressoud2} is still valid for $r=k$ when $c_1,c_2\to\infty$.

\subsection{A proof through our new Bailey lattice}

Replacing the use of Theorem~\ref{thm:baileylattice} by the unilateral version of our new Bailey lattice given in Theorem~\ref{thm:newbilatbaileylattice} gives the following sequence: iterate $r-i$ times Theorem~\ref{thm:baileylemma}, then use the unilateral version of Theorem~\ref{thm:newbilatbaileylattice}, and finally $i-1$ times Theorem~\ref{thm:baileylemma} with $a$ replaced by $a/q$. This yields a final Bailey pair relative to $a/q$ to which we apply~\eqref{bp} with $a$ replaced by $a/q$. This is summarised in the following result, to be compared with~\cite[Theorem 3.1]{AAB} (equivalently the unilateral version of Theorem~\ref{thm:bilatncsqbaileylattice}).

\begin{Theorem}\label{thm:ncsqnewbaileylattice}
If $(\alpha_n, \beta_n)$ is a Bailey pair relative to $a$, then for all integers $0\leq i\leq r$ and~${n\geq0}$, we have
\begin{gather}
\sum_{s_1\geq\dots\geq s_{r}\geq0}\frac{a^{s_1+\dots+s_r}q^{s_i+\dots+s_{r}}\beta_{s_r}}{(\rho_1\sigma_1)^{s_1}
\cdots(\rho_r\sigma_r)^{s_r}}\frac{(\rho_1,\sigma_1)_{s_1}\cdots(\rho_r,\sigma_r)_{s_r}}{(q)_{n-s_1}(q)_{s_1-s_2}\dots(q)_{s_{r-1}-s_r}}\nonumber\\
\qquad{}\times\frac{(a/\rho_1\sigma_1)_{n-s_1}(a/\rho_2\sigma_2)_{s_1-s_2}\cdots(a/\rho_i\sigma_i)_{s_{i-1}-s_i}}{(a/\rho_1,a/\sigma_1)_{n}(a/\rho_2,a/\sigma_2)_{s_1}
\cdots(a/\rho_i,a/\sigma_i)_{s_{i-1}}}\nonumber\\
\qquad{}\times
\frac{(aq/\rho_{i+1}\sigma_{i+1})_{s_i-s_{i+1}}\cdots(aq/\rho_r\sigma_r)_{s_{r-1}-s_r}}{(aq/\rho_{i+1},aq/\sigma_{i+1})_{s_i}
\cdots(aq/\rho_r,aq/\sigma_r)_{s_{r-1}}}\nonumber\\
\phantom{\qquad{}\times}{}=\frac{\alpha_0}{(q)_n(a)_n}+\sum_{j=1}^{n}
\frac{(\rho_1,\sigma_1,\dots,\rho_i,\sigma_i)_j(\rho_1\sigma_1\cdots\rho_i\sigma_i)^{-j}a^{ij}(1-a)}{(q)_{n-j}(a)_{n+j} (a/\rho_1,a/\sigma_1,\dots,a/\rho_i,a/\sigma_i)_j}\nonumber\\
\phantom{\qquad{}\times=}{}\times\left(\frac{(\rho_{i+1},\sigma_{i+1},\dots,\rho_r,\sigma_r)_j(\rho_{i+1}\sigma_{i+1}
\cdots\rho_r\sigma_r)^{-j}(aq)^{(r-i)j}q^j\alpha_j}{(aq/\rho_{i+1},aq/\sigma_{i+1},\dots,aq/\rho_r,aq/\sigma_r)_j\bigl(1-aq^{2j}\bigr)}\right.\nonumber\\
\phantom{\qquad{}\times=}{}\left.- \frac{(\rho_{i+1},\sigma_{i+1},\dots,\rho_r,\sigma_r)_{j-1}(\rho_{i+1}\sigma_{i+1}\cdots\rho_r\sigma_r)^{-j+1} (aq)^{(r-i)(j-1)}q^{j-1}\alpha_{j-1}}{(aq/\rho_{i+1},aq/\sigma_{i+1},
\dots,aq/\rho_r,aq/\sigma_r)_{j-1}\bigl(1-aq^{2j-2}\bigr)}\right).\!\!\!\label{ncsqnewlattice}
\end{gather}
\end{Theorem}

{\samepage Now let $n\to\infty$ in~\eqref{ncsqnewlattice} and simplify the factor $(q)_\infty^{-1}$ appearing on both sides, and rewrite the right-hand side by shifting the index $j$ to $j+1$ in the summation involving $\alpha_{j-1}$:
\begin{gather}
\sum_{s_1\geq\dots\geq s_{r}\geq0}\frac{a^{s_1+\dots+s_r}q^{s_i+\dots+s_{r}}\beta_{s_r}}{(\rho_1\sigma_1)^{s_1}\cdots(\rho_r\sigma_r)^{s_r}}\frac{(\rho_1,\sigma_1)_{s_1}
\cdots(\rho_r,\sigma_r)_{s_r}}{(q)_{s_1-s_2}\cdots(q)_{s_{r-1}-s_r}}\frac{(a/\rho_2\sigma_2)_{s_1-s_2}\cdots(a/\rho_i\sigma_i)_{s_{i-1}-s_i}} {(a/\rho_2,a/\sigma_2)_{s_1}
\cdots(a/\rho_i,a/\sigma_i)_{s_{i-1}}}\nonumber\\
\qquad{}\times\frac{(aq/\rho_{i+1}\sigma_{i+1})_{s_i-s_{i+1}}\cdots(aq/\rho_r\sigma_r)_{s_{r-1}-s_r}}{(aq/\rho_{i+1},aq/\sigma_{i+1})_{s_i}
\cdots(aq/\rho_r,aq/\sigma_r)_{s_{r-1}}}\nonumber\\
\phantom{\qquad{}\times}{}
=\frac{(a/\rho_1,a/\sigma_1)_\infty}{(a,a/\rho_1\sigma_1)_\infty}\sum_{j\geq0}\frac{1-a}{1-aq^{2j}}\nonumber\\
\phantom{\qquad{}\times=}{}\times
\frac{(\rho_1,\sigma_1,\dots,\rho_r,\sigma_r)_j(\rho_1\sigma_1\cdots\rho_r\sigma_r)^{-j}a^{rj}q^{(r-i+1)j}\alpha_j}{(a/\rho_1,a/\sigma_1,
\dots,a/\rho_i,a/\sigma_i,aq/\rho_{i+1},aq/\sigma_{i+1},\dots,aq/\rho_r,aq/\sigma_r)_j}\nonumber\\
\phantom{\qquad{}\times=}{}\times\left( 1-\frac{a^{i}}{\rho_1\sigma_1\cdots\rho_i\sigma_i}\right.\nonumber\\
\phantom{\qquad{}\times=}{}\left.\times\frac{\bigl(1-\rho_1q^j\bigr)\bigl(1-\sigma_1q^j\bigr)
\cdots\bigl(1-\rho_iq^j\bigr)\bigl(1-\sigma_iq^j\bigr)}{\bigl(1-aq^j/\rho_1\bigr)\bigl(1-aq^j/\sigma_1\bigr)
\cdots\bigl(1-aq^j/\rho_i\bigr)\bigl(1-aq^j/\sigma_i\bigr)}\right).\label{newlattice}
\end{gather}}%

In~\eqref{newlattice}, replace $r$ by $r-1$, and use the Bailey pair obtained from the unit Bailey pair~\eqref{ubp} by one iteration of the Bailey lemma given in Theorem~\ref{thm:baileylemma}. This yields for $0\leq i\leq r-1$,
\begin{gather}
\sum_{s_1\geq\dots\geq s_{r-1}\geq0}\frac{a^{s_1+\dots+s_{r-1}}q^{s_{i}+\dots+s_{r-1}}}{(\rho_1\sigma_1)^{s_1}\cdots(\rho_{r-1}\sigma_{r-1})^{s_{r-1}}}
\frac{(aq/\rho\sigma)_{s_{r-1}}}{(q,aq/\rho,aq/\sigma)_{s_{r-1}}}\frac{(\rho_1,\sigma_1)_{s_1}\cdots(\rho_{r-1}, \sigma_{r-1})_{s_{r-1}}}{(q)_{s_1-s_2}\cdots(q)_{s_{r-2}-s_{r-1}}}\nonumber\\
\qquad{}\times\frac{(a/\rho_2\sigma_2)_{s_1-s_2}\cdots(a/\rho_i\sigma_i)_{s_{i-1}-s_i}}{(a/\rho_2,a/\sigma_2)_{s_1}
\cdots(a/\rho_i,a/\sigma_i)_{s_{i-1}}}\frac{(aq/\rho_{i+1}\sigma_{i+1})_{s_i-s_{i+1}}\cdots(aq/\rho_{r-1} \sigma_{r-1})_{s_{r-2}-s_{r-1}}}{(aq/\rho_{i+1},aq/\sigma_{i+1})_{s_i}
\cdots(aq/\rho_{r-1},aq/\sigma_{r-1})_{s_{r-2}}}\nonumber\\
\phantom{\qquad{}\times}{}=\frac{(a/\rho_1,a/\sigma_1)_\infty}{(a,a/\rho_1\sigma_1)_\infty}
\sum_{j\geq0}\frac{(-1)^j(\rho_1,\sigma_1,\dots,\rho_{r-1},\sigma_{r-1},\rho,\sigma,a)_j}{(a/\rho_1,a/\sigma_1,\dots,a/\rho_i,a/\sigma_i)_j}\nonumber\\
\phantom{\qquad{}\times=}{}\times\frac{(\rho_1\sigma_1\cdots\rho_{r-1}\sigma_{r-1}\rho\sigma)^{-j}a^{rj}q^{(r-i+1)j+j(j-1)/2}}{(aq/\rho_{i+1},aq/\sigma_{i+1},
\dots,aq/\rho_{r-1},aq/\sigma_{r-1},aq/\rho,aq/\sigma,q)_j}\left(1-\frac{a^{i}}{\rho_1\sigma_1\cdots\rho_i\sigma_i}\right.\nonumber\\
\phantom{\qquad{}\times=}{}\times\left.\frac{\bigl(1-\rho_1q^j\bigr)\bigl(1-\sigma_1q^j\bigr)
\cdots\bigl(1-\rho_iq^j\bigr)\bigl(1-\sigma_iq^j\bigr)}{\bigl(1-aq^j/\rho_1\bigr)\bigl(1-aq^j/\sigma_1\bigr) \cdots\bigl(1-aq^j/\rho_i\bigr)\bigl(1-aq^j/\sigma_i\bigr)}\right).\label{newlattice2}
\end{gather}

Next, in~\eqref{newlattice2}, take $\sigma_1,\rho_j,\sigma_j\to\infty$ for $j=i+1,\dots,r-1$, which yields
\begin{gather}
\sum_{s_1\geq\dots\geq s_{r-1}\geq0}(-1)^{s_1}\frac{a^{s_1+\dots+s_{r-1}}q^{s_1^2/2+s_{i+1}^2+\dots+s_{r-1}^2-s_1/2+s_i}}{(\rho_1)^{s_1}(\rho_{2}\sigma_{2})^{s_{2}} \cdots(\rho_{i}\sigma_{i})^{s_{i}}}
\frac{(aq/\rho\sigma)_{s_{r-1}}}{(q,aq/\rho,aq/\sigma)_{s_{r-1}}}\nonumber\\
\qquad{}\times\frac{(\rho_1)_{s_1}(\rho_{2},\sigma_{2})_{s_{2}}\cdots(\rho_{i},\sigma_{i})_{s_{i}}}{(q)_{s_1-s_2} \cdots(q)_{s_{r-2}-s_{r-1}}}\frac{(a/\rho_2\sigma_2)_{s_1-s_2}
\cdots(a/\rho_i\sigma_i)_{s_{i-1}-s_i}}{(a/\rho_2,a/\sigma_2)_{s_1}\cdots(a/\rho_i,a/\sigma_i)_{s_{i-1}}}\nonumber\\
\phantom{\qquad{}\times}{}=\frac{(a/\rho_1)_\infty}{(a)_\infty}\sum_{j\geq0}\frac{(\rho_1,\rho_2,\sigma_2,\dots,\rho_{i},\sigma_{i},\rho,\sigma,a)_j
(\rho_1\rho_2\sigma_2\cdots\rho_{i}\sigma_{i}\rho\sigma)^{-j}a^{rj}q^{(r-i)j^2+j}}{(a/\rho_1,a/\rho_2,a/\sigma_2, \dots,a/\rho_i,a/\sigma_i,aq/\rho,aq/\sigma,q)_j}\nonumber\\
\phantom{\qquad{}\times=}{}\times
\bigg(1+\frac{a^{i}q^{j}}{\rho_1\rho_2\sigma_2\cdots\rho_i\sigma_i}\nonumber\\
\phantom{\qquad{}\times=}{}\times\frac{\bigl(1-\rho_1q^j\bigr)\bigl(1-\rho_2q^j\bigr)\bigl(1-\sigma_2q^j\bigr)\cdots\bigl(1-\rho_iq^j\bigr)
\bigl(1-\sigma_iq^j\bigr)}{\bigl(1-aq^j/\rho_1\bigr)
\bigl(1-aq^j/\rho_2\bigr)\bigl(1-aq^j/\sigma_2\bigr)\cdots\bigl(1-aq^j/\rho_i\bigr)\bigl(1-aq^j/\sigma_i\bigr)}\bigg).\label{newlattice3}
\end{gather}

Replacing $k$ by $r$ and $r$ by $i$, Bressoud's formula~\eqref{Bressoud2} becomes
\begin{equation*}
F_{2i-1,r,i}(c_1,c_2;b_1,\dots,b_{2i-1};a;q)=G_{2i-1,r,i}(c_1,c_2;b_1,\dots,b_{2i-1};a;q),
\end{equation*}
which corresponds to~\eqref{newlattice3} by taking $c_1=\rho$, $c_2=\sigma$, $b_1=\rho_1$, $b_2=\rho_2$, $b_{2i-1}=\sigma_2$, $\dots$, $b_{i}=\rho_i$, $b_{i+1}=\sigma_i$. As~\eqref{newlattice3} is valid for $0\leq i\leq r-1$, we conclude that both formulas~\eqref{eq:bressoud} and~\eqref{eq:bressoudlim} of Bressoud's theorem are special cases of our Bailey lattice.

In the first place, we tried to use the classical Bailey lattice of~\cite[Theorem 3.1]{AAB} (or the unilateral version in Theorem~\ref{thm:bilatncsqbaileylattice}) instead of Theorem~\ref{thm:ncsqnewbaileylattice}, and saw that to recover Bressoud's formula~\eqref{Bressoud2}, one has to follow the same lines as above. We came up with the following formula instead of~\eqref{newlattice3}:
\begin{gather}
\sum_{s_1\geq\dots\geq s_{r-1}\geq0}(-1)^{s_1}\frac{a^{s_1+\dots+s_{r-1}}q^{s_1^2/2+s_{i+1}^2+\dots+s_{r-1}^2-s_1/2}}{(\rho_1)^{s_1}(\rho_{2}\sigma_{2})^{s_{2}} \cdots(\rho_{i}\sigma_{i})^{s_{i}}}
\frac{(aq/\rho\sigma)_{s_{r-1}}}{(q,aq/\rho,aq/\sigma)_{s_{r-1}}}\nonumber\\
\qquad{}\times\frac{(\rho_1)_{s_1}(\rho_{2},\sigma_{2})_{s_{2}}\cdots(\rho_{i},\sigma_{i})_{s_{i}}}{(q)_{s_1-s_2} \cdots(q)_{s_{r-2}-s_{r-1}}}\frac{(a/\rho_2\sigma_2)_{s_1-s_2}
\cdots(a/\rho_i\sigma_i)_{s_{i-1}-s_i}}{(a/\rho_2,a/\sigma_2)_{s_1}\cdots(a/\rho_i,a/\sigma_i)_{s_{i-1}}}\nonumber\\
\phantom{\qquad{}\times}=\frac{(a/\rho_1)_\infty}{(a)_\infty}
\sum_{j\geq0}\frac{(\rho_1,\rho_2,\sigma_2,\dots,\rho_{i},\sigma_{i},\rho,\sigma,a)_j
(\rho_1\rho_2\sigma_2\cdots\rho_{i}\sigma_{i}\rho\sigma)^{-j}a^{rj}q^{(r-i)j^2}}{(a/\rho_1,a/\rho_2,a/\sigma_2,\dots,a/\rho_i,a/\sigma_i,aq/\rho,aq/\sigma,q)_j}
\nonumber\\
\phantom{\qquad\times=}\times\bigg(1+\frac{a^{i+1}q^{3j}}{\rho_1\rho_2\sigma_2\cdots\rho_i\sigma_i}\nonumber\\
\phantom{\qquad\times=}\times\frac{\bigl(1-\rho_1q^j\bigr)\bigl(1-\rho_2q^j\bigr)\bigl(1-\sigma_2q^j\bigr) \cdots\bigl(1-\rho_iq^j\bigr)\bigl(1-\sigma_iq^j\bigr)}{\bigl(1-aq^j/\rho_1\bigr)
\bigl(1-aq^j/\rho_2\bigr)\bigl(1-aq^j/\sigma_2\bigr)\cdots\bigl(1-aq^j/\rho_i\bigr)\bigl(1-aq^j/\sigma_i\bigr)}\bigg).\label{lattice3}
\end{gather}

Therefore, we could only prove the special case
\begin{gather*}
\lim_{b_{i+1},b_{i+2}\to\infty}F_{2i+1,r,i+1}(c_1,c_2;b_1,\dots,b_{2i+1};a;q)\\
\qquad=\lim_{b_{i+1},b_{i+2}\to\infty}G_{2i+1,r,i+1}(c_1,c_2;b_1,\dots,b_{2i+1};a;q)
\end{gather*}
of Bressoud's formula~\eqref{Bressoud2} in which one takes $k=r$, $r=i+1$, $c_1=\rho$, $c_2=\sigma$, $b_1=\rho_1$, $b_2=\rho_2$, $b_{2i+1}=\sigma_2$, $\dots$, $b_{i}=\rho_i$, $b_{i+3}=\sigma_i$.

But we could not derive the most general identity~\eqref{Bressoud2} of Bressoud in this way.

%%%%%%%%%%%%%%%%%%
\subsection{Special cases}
%%%%%%%%%%%%%%%%%%

The case $\lambda=1$ obtained from~\eqref{Bressoud2} by taking $b_j\to\infty$ for all $j\geq2$ (and replacing $k$ by $r$ and $r$ by $i$), exactly corresponds to~\eqref{lattice3} in which one takes $\rho_j,\sigma_j\to\infty$ for all $j\geq2$ and $\rho_1=b_1,\rho=c_1$, $\sigma=c_2$:
\begin{gather}
\sum_{s_1\geq\dots\geq s_{r-1}\geq0}(-1)^{s_1}\frac{a^{s_1+\dots+s_{r-1}}q^{s_1^2/2+s_2^2+\dots+s_{r-1}^2-s_{1}/2-s_{2}-\dots-s_{i-1}}}{(q)_{s_1-s_2} \cdots(q)_{s_{r-2}-s_{r-1}}(q)_{s_{r-1}}}
\frac{(b_1)_{s_1}}{b_1^{s_1}}\frac{(aq/c_1c_2)_{s_{r-1}}}{(aq/c_1,aq/c_2)_{s_{r-1}}}\nonumber\\
\qquad=\frac{(a/b_1)_\infty}{(a)_\infty}\sum_{j\geq0}a^{rj}q^{(r-1)j^2+(2-i)j}\left(1+\frac{a^{i}q^{(2i-1)j}}{b_1}\frac{1-b_1q^j}{1-aq^j/b_1}\right)\nonumber\\
\phantom{\qquad=}{}\times\frac{(b_1,c_1,c_2,a)_j(b_1c_1c_2)^{-j}}{(a/b_1,aq/c_1,aq/c_2,q)_j}.\label{lambda1}
\end{gather}

Note that we obtain the exact same formula with $i$ replaced by $i+1$ by taking similar limits in~\eqref{lattice3}.

Obviously, the case $\lambda=0$ of Bressoud's result is obtained from~\eqref{lambda1} by taking $b_1\to\infty$. Moreover all special case (3.2)--(3.7) in~\cite{Br80} are consequences of the latter $\lambda=0$ case, with the choices $(c_1\to\infty,c_2\to\infty,a=q)$, $(c_1\to\infty,c_2\to\infty,a=1)$, $(c_1=-q,c_2\to\infty,a=q)$, $(c_1=-1,c_2\to\infty,a=1)$, $(q\to q^2,c_1=-q,c_2\to\infty,a=1)$, and $(q\to q^2,c_1=-q,c_2\to\infty,a=q^2)$, respectively.
The two other special cases (3.8) and (3.9) of~\cite{Br80} are obtained from~\eqref{lambda1} with $\bigl(q\to q^2,c_1\to\infty,c_2\to\infty,b_1=-q,a=q^2\bigr)$ and $\bigl(q\to q^2,c_1=-q^2,c_2\to\infty,b_1=-q,a=q^2\bigr)$, respectively.

Therefore, we can conclude that all special cases of Bressoud's theorem exhibited in~\cite{Br80} (giving Andrews--Gordon, Bressoud, and G\"ollnitz--Gordon type identities) are consequences of both unilateral Bailey lattices, the classical one and the new one. This is not surprising by Remark~\ref{rk:coronewbilatbaileylattice}.

\section{Conclusion and open problems}
We saw several applications of our key lemmas and bilateral Bailey lattices throughout the paper. We conclude with a few possible further applications suggested by the referees.
\begin{itemize}\itemsep=0pt
\item In \cite{P}, Paule gave a way to iterate Bailey pairs different from the one of Andrews \cite{A1}, essentially creating in each step a
new Bailey pair with one free parameter $b$ (instead
of the two free parameters $\rho$ and $\sigma$ of Andrews) where in the underlying
simplification of the double sum one only requires the $q$-Chu--Vandermonde
summation (instead of the $q$-Pfaff--Saalsch\"utz summation). Despite of the simpler lemma, iteration still gives the full Andrews--Gordon and
Bressoud identities, and also a number of new identities for multisums. Could some of these identities be reproved or extended using the tools of the present paper?

\item Theorem \ref{thm:bilatncsqbaileylattice} seems reminiscent of a~theorem of Milne \cite[Theorem 1.7]{Mil80}, also related to the Bailey machinery. Would it be possible to use our techniques to prove this theorem directly? Milne originally proved it by applying Ismail's analytic continuation method to Andrew's multisum extension of the Watson transformation.

\item In a recent paper \cite{Sch23}, Schlosser obtained some other bilateral identities of the Rogers--Ramanujan type (see, for example, Theorem 2.1 or Theorem
4.6 in his paper). Can one find a bilateral Bailey pair that would yield Schlosser's results?

\item Recently Warnaar \cite{W23} gave an extension of the $A_2$ Bailey chain into a tree, proving Rogers--Ramanujan type $q$-series identities related to characters of the affine Lie algebra~\smash{$A_2^{(1)}$}. Can the $A_2$ Bailey chain or Warnaar's generalisation be bilateralised? If so, could this be extended to the $A_n$ or $C_n$ Bailey chains?

\item Related to the previous question, there already exist $A_n$ and $C_n$ Bailey lemmas, which are however essentially different from Warnaar's $A_2$ case, and could be named $A_n$ and $C_n$ Bailey lemmas of ``Milne type'' (see for instance the work of Milne and Lilly in~\cite{ML92, ML93}). Can one find ``Milne type'' $A_n$ and $C_n$ extensions of the results in the present paper? In~\cite{BS}, Bhatnagar and Schlosser find $A_n$ and $C_n$ elliptic Bailey lemmas, whose $p=0$ cases yield $A_n$ and $C_n$ well-poised Bailey lemmas.
\end{itemize}

\subsection*{Acknowledgements}

The authors are grateful to Jeremy Lovejoy for very helpful comments on an earlier version of this paper. The authors also thank the anonymous referees for their very careful read of the paper and for their great suggestions for improvement and future research directions.
The authors are partially funded by the ANR COMBIN\'e
ANR-19-CE48-0011. JD is funded by the SNSF Eccellenza grant number PCEFP2 202784.

\pdfbookmark[1]{References}{ref}
\LastPageEnding

\end{document}